\documentclass[psamsfonts]{amsart}  % proc-l

\usepackage[T1]{fontenc}
\usepackage{amssymb}
\usepackage[dvips]{graphicx}

\usepackage{graphicx}
\usepackage{amssymb}                       
\usepackage{amsmath}
\usepackage{color}
\usepackage{times}

\usepackage[unicode,bookmarks,colorlinks]{hyperref}
\hypersetup{
   linkcolor=brickred,
}

\definecolor{mahogany}{cmyk}{0, 0.77, 0.87, 0}
\definecolor{salmon}{cmyk}{0, 0.53, 0.38, 0}
\definecolor{melon}{cmyk}{0, 0.46, 0.50, 0}
\definecolor{yellowgreen}{cmyk}{0.44, 0, 0.74, 0}
\definecolor{brickred}{cmyk}{0, 0.89, 0.94, 0.28}
\definecolor{OliveGreen}{cmyk}{0.64, 0, 0.95, 0.40}
\definecolor{RawSienna}{cmyk}{0, 0.72, 1.0, 0.45}
\definecolor{ZurichRed}{rgb}{1, 0, 0} % Red of svgnames

\usepackage{fancyhdr}
\usepackage{amsmath,amstext,amssymb,amsopn,amsthm}
\usepackage{amsmath,amssymb,amsthm}
\usepackage[mathscr]{eucal}

\reversemarginpar

\definecolor{rb}{rgb}{0.1,0.2, 0.7}

\begin{document}

%\newtheorem{thm}{Theorem}
%\numberwithin{thm}{Theorem}
\newtheorem{lemma}[thm]{Lemma}
%\newtheorem{corr}[thm]{Corollary}[sections]
%\numberwithin{corollary}[thm]{section}
\newtheorem{proposition}{Proposition}
\newtheorem{theorem}{Theorem}[section]
\newtheorem{deff}[thm]{Definition}
\newtheorem{case}[thm]{Case}
%\numberwithin{deff}{section}
\newtheorem{prop}[thm]{Proposition}
%\numberwithin{equation}{subsection}
%\numberwithin{equation}{section}
\newtheorem{example}{Example}

\newtheorem{corollary}{Corollary}

\theoremstyle{definition}
\newtheorem{remark}{Remark}

\numberwithin{equation}{section}
\numberwithin{definition}{section}
%\numberwithin{problem}{section}
\numberwithin{corollary}{section}
%\numberwithin{proposition}{subsection}

\numberwithin{theorem}{section}

\numberwithin{remark}{section}
\numberwithin{example}{section}
\numberwithin{proposition}{section}

\newcommand{\gap}{\lambda_{2,D}^V-\lambda_{1,D}^V}
\newcommand{\gapR}{\lambda_{2,R}-\lambda_{1,R}}
\newcommand{\bD}{\mathrm{I\! D\!}}
\newcommand{\calD}{\mathcal{D}}
\newcommand{\calA}{\mathcal{A}}

\newcommand{\conjugate}[1]{\overline{#1}}
\newcommand{\abs}[1]{\left| #1 \right|}
\newcommand{\cl}[1]{\overline{#1}}
\newcommand{\expr}[1]{\left( #1 \right)}
\newcommand{\set}[1]{\left\{ #1 \right\}}

\newcommand{\calC}{\mathcal{C}}
\newcommand{\calE}{\mathcal{E}}
\newcommand{\calF}{\mathcal{F}}
\newcommand{\Rd}{\mathbb{R}^d}
\newcommand{\BR}{\mathcal{B}(\Rd)}
\newcommand{\R}{\mathbb{R}}
\newcommand{\T}{\mathbb{T}}
\newcommand{\D}{\mathbb{D}}

\newcommand{\al}{\alpha}
\newcommand{\RR}[1]{\mathbb{#1}}
\newcommand{\bR}{\mathrm{I\! R\!}}
\newcommand{\ga}{\gamma}
\newcommand{\om}{\omega}
\newcommand{\A}{\mathbb{A}}
\newcommand{\bH}{\mathbb{H}}

\newcommand{\bb}[1]{\mathbb{#1}}
\newcommand{\bI}{\bb{I}}
\newcommand{\bN}{\bb{N}}

\newcommand{\uS}{\mathbb{S}}
\newcommand{\M}{{\mathcal{M}}}
\newcommand{\calB}{{\mathcal{B}}}

\newcommand{\W}{{\mathcal{W}}}

\newcommand{\m}{{\mathcal{m}}}

\newcommand {\mac}[1] { \mathbb{#1} }

\newcommand{\bC}{\Bbb C}

\newtheorem{rem}[theorem]{Remark}
\newtheorem{dfn}[theorem]{Definition}
\theoremstyle{definition}
\newtheorem{ex}[theorem]{Example}
\numberwithin{equation}{section}

\newcommand{\Pro}{\mathbb{P}}
\newcommand\F{\mathcal{F}}
\newcommand\E{\mathbb{E}}
\newcommand\e{\varepsilon}
\def\H{\mathcal{H}}
\def\t{\tau}

\title[Burkholder's function]{Burkholder's function\\ and a weighted $L^2$ bound for stochastic integrals}

\author{Rodrigo Ba\~nuelos}\thanks{R. Ba\~nuelos is supported in part  by NSF Grant
  \#1403417-DMS}
\address{Department of Mathematics, Purdue University, West Lafayette, IN 47907, USA}
\email{banuelos@math.purdue.edu}
\author{Micha\l { }Brzozowski}\thanks{M. Brzozowski is supported in part by the NCN grant DEC-2014/14/E/ST1/00532.}
\address{Department of Mathematics, Informatics and Mechanics, University of Warsaw, Banacha 2, 02-097 Warsaw, Poland}
\email{M.Brzozowski@mimuw.edu.pl}
\author{Adam Os\c ekowski}\thanks{A. Os\c ekowski is supported in part by the NCN grant DEC-2014/14/E/ST1/00532.}
\address{Department of Mathematics, Informatics and Mechanics, University of Warsaw, Banacha 2, 02-097 Warsaw, Poland}
\email{A.Osekowski@mimuw.edu.pl}

\subjclass[2010]{Primary: 60G42. Secondary: 60G44.}

\keywords{martingale, martingale transform, Bellman function, differential subordination}

\begin{abstract}
Let $X$ be a continuous-path martingale and let $Y$ be a stochastic integral, with respect to $X$, of some predictable process with values in $[-1,1]$. We provide an explicit formula for Burkholder's function associated with the weighted $L^2$ bound
$$ \|Y\|_{L^2(W)}\lesssim [w]_{A_2}\|X\|_{L^2(W)}.$$
\end{abstract}

\maketitle

\section{Introduction}
Suppose that $(\Omega,\F,\mathbb{P})$ is a complete probability space equipped with a right-continuous filtration $(\F_t)_{t\geq 0}$, a nondecreasing family of sub-$\sigma$-algebras of $\F$. Throughout the paper, we will assume that all adapted martingales have continuous paths; for example, this is the case if $(\F_t)_{t\geq 0}$ is a Brownian filtration. Let $X$ be an adapted martingale and let $X^*=\sup_{t\geq 0}X_t$, $|X|^*=\sup_{t\geq 0}|X_t|$ denote the associated one- and two-sided maximal functions. In what follows, $\langle X\rangle$ will stand for the corresponding (skew) square bracket: see Dellacherie and Meyer \cite{DM} for the definition and basic properties of this object. Next, suppose that $Y$ is the stochastic integral, with respect to $X$, of some predictable process $H$ which takes values in $[-1,1]$:
$$ Y_t=H_0X_0+\int_{0+}^t H_s\mbox{d}X_s.$$
The question about the comparison of the sizes of $X$ and $Y$ has gathered a lot of interest in the literature. See e.g. \cite{Bu1,Bu2,Bu3,Bu4} and consult the monograph \cite{O1}. In addition, such stochastic inequalities have found numerous applications in harmonic analysis, where they can be used, among other things, in the study of $L^p$ boundedness of wide classes of Fourier multipliers \cite{BB,BO11,BO12}. We have the following celebrated result proved in \cite{Bu2}.

\begin{theorem}
If $X$, $Y$ are as above, then for each $1<p<\infty$ we have
\begin{equation}\label{LpB}
 \|Y\|_{L^p}\leq (p^*-1)\|X\|_{L^p},
\end{equation}
where $p^*=\max\{p,p/(p-1)\}$. For each $p$, the constant is the best possible.
\end{theorem}

There is a powerful method, invented by Burkholder, which allows the efficient study of general class of inequalities for martingales and their stochastic integrals. Roughly speaking, the approach enables to deduce the desired estimate from the existence of a certain special function enjoying appropriate concavity and size requirements. This method (also referred to as the Bellman function method) originates in the theory of optimal control, and has turned out to work also in much wider settings of harmonic analysis. See e.g. \cite{Be,NT,NTV,O1}. 

For example, in order to prove the above sharp $L^p$ estimate, Burkholder showed that it is enough to find a continuous function $B:\R^2\to \R$ satisfying
\begin{itemize}
\item[1$^\circ$] $B(x,y)\leq 0$ if $|y|\leq |x|$;
\item[2$^\circ$] $B(x,y)\geq |y|^p-(p^*-1)^p|x|^p$
\end{itemize}
and the following concavity type condition:
\begin{itemize}
\item[3$^\circ$] for any $x$, $y$ and any $h,\,k$ with $|k|\leq |h|$, the function $t\mapsto B(x+th,y+tk)$ is concave on $\R$.
\end{itemize}
See \cite{Bu1,Bu2,Bu4} or Chapter 4 in \cite{O1} for the relation of such a function to \eqref{LpB}, consult also Section 2 below. To complete the proof of the $L^p$ bound, Burkholder provides the explicit formula for $B$:
$$ B(x,y)=\alpha_p(|y|-(p^*-1)|x|)(|x|+|y|)^{p-1},$$
where $\alpha_p$ is a certain constant depending only on $p$. It turns out that this function can be applied in seemingly unrelated areas of mathematics. Namely, there is a deep and unexpected connection of $B$ with the geometric function theory, particularly with the theory of quasiconformal mappings, rank-one convex functionals and the properties of Beurling-Ahlfors operator: see \cite{AIPS0, AIPS,BMS,I1,I2} and consult the references therein. In other words, although the function $B$ originates in the probabilistic estimate \eqref{LpB}, its explicit formula is of independent interest and importance in contexts far and beyond martingale theory. 

The above observation was one of the motivations for our research. 
There is an interesting question concerning the explicit formula for a \emph{weighted} version of Burkholder's function $B$. Suppose that $W=(W_t)_{t\geq 0}$ is a weight, i.e., a nonnegative and uniformly integrable martingale. It is a usual convention to identify $W$ with its terminal variable $W_\infty$. For any $1\leq p<\infty$ and any weight $W$, we introduce the associated $L^p$ space as the class of all variables $f$ for which $ \|f\|_{L^p(W)}=\left(\int_\Omega |f|^pW\mbox{d}\mathbb{P}\right)^{1/p}<\infty.$ Given a martingale $X$ as above, we will also use the notation $\|X\|_{L^p(W)}=\sup_{t\geq 0}\|X_t\|_{L^p(W)}.$ 
Following Izumisawa and Kazamaki \cite{IK}, we say that $W$ is an $A_p$ weight (where $1<p<\infty$ is a fixed parameter), if the $A_p$ characteristic of $W$, given by the formula
$$ [W]_{A_p}=\sup_{t\geq 0} \bigg\|W_t\E(W^{-1/(p-1)}|\F_t)\bigg\|_\infty,$$
is finite. 
This is the probabilistic counterpart of the classical, analytic $A_p$ condition introduced by Muckenhoupt \cite{Mu} during the study of boundedness of the Hardy-Littlewood maximal operator on weighted spaces. 

With all the necessary definitions at hand, we can ask about the weighted analogue of \eqref{LpB}. Namely, for a given and fixed $1<p<\infty$ and a weight $W$, does there exist a constant $C_{p,W}$ such that we have
$$ \|Y\|_{L^p(W)}\leq C_{p,W}\|X\|_{L^p(W)}$$
for all martingales $X$, $Y$ such that $Y$ is the stochastic integral of $X$? It can be shown (cf. Domelevo and Petermichl \cite{DP}, Petermichl and Volberg \cite{PV}, Wittwer \cite{Wi}) that the answer is positive if and only if $W\in A_p$. Furthermore, one can show the following optimal factorization of the constant: we have $C_{p,W}\leq c_p[W]_{A_p}^{\max\{1,1/(p-1)\}}$, where $c_p$ depends only on $p$ and the exponent $\max\{1,1/(p-1)\}$ is the best possible. Such extraction of the optimal dependence of the constant on the weight characteristic has gained a lot of interest in the recent literature. For most classical operators in harmonic analysis such an extraction has been carried out successfully: see e.g. \cite{Buc,Hyt,L,Le1} and consult the references therein. 

Coming back to the context of martingale transforms, the above discussion shows that
\begin{equation}\label{L2D}
 \|Y\|_{L^p(W)}\leq c_p[W]_{A_p}^{\max\{1,1/(p-1)\}}\|X\|_{L^p(W)},\qquad 1<p<\infty.
\end{equation}
Straightforward extrapolation techniques (see e.g. \cite{D} or, in the above probabilistic context, \cite{BO2}) show that it is enough to study the above bound for the case $p=2$ only: 
\begin{equation}\label{L2DP}
 \|Y\|_{L^2(W)}\leq c_2[W]_{A_2}\|X\|_{L^2(W)}.
\end{equation}
This reduction was used in \cite{DP,PV,Wi} to establish the above weighted $L^p$ bound. We want to emphasize here that due to this fact, we only construct the Burkholder's function associated with the weighted $L^2$ estimate (as stated in the abstract) and not for the weighted $L^p$ estimate.  To show \eqref{L2DP}, a duality and a number of complicated Bellman functions (involving six variables) were applied. 
There is a natural question whether the $L^2$ bound \eqref{L2DP} can be established directly, in the spirit of Burkholder's approach described earlier. The presence of $A_2$ weights forces the introduction of two additional arguments and hence the problem reduces to the construction of an explicit function of \emph{four} variables, enjoying appropriate concavity and size conditions similar to 1$^\circ$-3$^\circ$ above (see Section 2 below for details). The main result of this paper is to give a positive answer to this question. Interestingly, as an immediate by-product, this special function will allow us to obtain a stronger, maximal estimate stated below.

\begin{theorem} Suppose that $W$ is an $A_{p}$ weight, $X$ is a martingale and $Y$ is a stochastic integral, with respect to $X$, of some predictable process $X$ taking values in $[-1,1]$. Then for any $1<p<\infty$ there is a finite constant $C_{p}$ depending only on $p$ such that
\begin{align}\label{MainTheorem}
\|\,|Y|^*\|_{L^p(W)}\leq C_p[W]_{A_p}^{\max\{1/(p-1),1\}}\|X\|_{L^p(W)}.
\end{align}
The exponent $\max\{1/(p-1),1\}$ is the best possible.
\end{theorem}

As we will see, the function $B$ we provide has quite a complicated structure (which should be compared to its trivial unweighted counterpart: $B(x,y)=y^2-x^2$). Of course, this increased difficulty is not surprising: in the light of the extrapolation method mentioned above, the weighted $L^2$ bound implies the validity of \eqref{L2D} and hence the corresponding Burkholder's function carries all the information about all $L^p$ estimates in the weighted context. We would like to finish the discussion with a terminological remark. Namely, the function $B$ constructed in this paper yields the constant $c_2$ in \eqref{L2DP} which is not optimal. Therefore, in the language used in the Bellman function theory, one could call $B$ a \emph{supersolution} corresponding to \eqref{L2DP}.

The remaining part of the paper is split into two sections. In Section 2 we explain the relation between Burkholder's function and the validity of \eqref{L2DP}. Section 3 contains the explicit construction of the special function and the proof of \eqref{MainTheorem}.

\section{Burkholder's method}

Let us start with the following useful interpretation of $A_p$ weights, valid for $1<p<\infty$. Fix such a weight $W$ and suppose that $c\geq [W]_{A_p}$. In particular, the finiteness of the $A_p$ characteristic implies the integrability of the function $W^{1/(1-p)}$ and we may consider the associated martingale $V=(V_t)_{t\geq 0}$ given by $V_t=\E(W^{1/(1-p)}|\F_t)$, $t\geq 0$. Note that Jensen's inequality implies $W_t V_t^{p-1}\geq 1$ almost surely for any $t\geq 0$ and, in addition, the $A_p$ condition is equivalent to the reverse bound
$$ W_t V_t^{p-1}\leq [W]_{A_p}\qquad \mbox{with probability }1.$$
In other words, an $A_p$ weight of characteristic equal to $c$ gives rise to a two-dimensional martingale $(W,V)$ taking values in the domain
$$ \mathcal{D}_c=\{(w,v)\in (0,\infty)\times (0,\infty)\,:\,1\leq wv^{p-1}\leq c\}.$$
Note that this martingale terminates at the lower boundary of this domain: $W_\infty V_\infty^{p-1}=1$ almost surely. Actually, the implication can be reversed.  Given a pair $(W,V)$ taking values in $\mathcal{D}_c$ and terminating at the set $wv^{p-1}=1$,  one easily checks that its first coordinate is an $A_p$ weight with $[W]_{A_p}\leq c$.

Let $c\geq 1$ be a fixed parameter. Suppose that $G:\R^2\times \mathcal{D}_c\to \R$ is a given Borel function and assume that we are interested in showing that 
\begin{equation}\label{wewant}
\E G(X_t,Y_t,W_t,V_t)\leq 0,\qquad t\geq 0.
\end{equation}
Here $(X,Y)$ is an arbitrary pair of martingales such that $Y$ is the stochastic integral, with respect to $X$, of some predictable process with values in $[-1,1]$, and $(W,V)$ is a pair associated with some $A_p$ weight of characteristic not bigger than $c$. A key to handle this problem is to consider a $C^2$ function $B:\R^2\times \mathcal{D}_{c} \to \mathbb{R}$ which satisfies the following properties:
\begin{itemize}
\item[1$^\circ$] (Initial condition) We have $B(x,y,{w},{v}) \leq 0$ if $|y|\leq|x|$ and $1\leq \mathtt{wv}^{p-1} \leq c$.
\item[2$^\circ$] (Majorization property) We have $B\geq G$ on $\R^2\times \mathcal{D}_c$.
\item[3$^\circ$] (Concavity-type property) For any $(x,y,{w},{v})\in \R^2\times \mathcal{D}_{c}$ and $d,e,r,s \in \mathbb{R}$ satisfying $|e|\leq |d|$, the function
\begin{equation*}
\xi_{B}(t) := B(x+td,y+te,{w}+tr,{v}+ts),
\end{equation*}
given for those $t$, for which $1 \leq ({w}+tr)({v}+ts) \leq c$, is locally concave.
\end{itemize}

The connection between the existence of such a function and the validity of \eqref{wewant} is described in the lemma below.

\begin{lemma}\label{MainLemma}
Let $1<p<\infty$ and $c\geq 1$ be fixed. If $B$ satisfies the conditions 1$^\circ$, 2$^\circ$ and 3$^\circ$, then the inequality \eqref{wewant} holds true for all $X,\,Y,\,W$ and $V$ as above.
\end{lemma}
\begin{proof}
The argument rests on  It\^o's formula. Consider an auxiliary process $Z=(X,Y,W,V)$. Since $B$ is of class $C^2$, we may write
$$ B(Z_t)=I_0+I_1+I_2/2,$$
where
\begin{align*}
I_0&=B(Z_0),\\
 I_1&=\int_{0+}^t B_x(Z_{u})\mbox{d}X_u+\int_{0+}^t B_y(Z_{u})\mbox{d}Y_u+\int_{0+}^t B_w(Z_{u})\mbox{d}W_u+\int_{0+}^t B_v(Z_{u})\mbox{d}V_u,\\
I_2&=\int_{0+}^t D^2B(Z_{u})\mbox{d}\langle Z\rangle_u.
\end{align*}
Here $D^2B$ is the Hessian matrix of $B$ and in the definition of $I_2$ we have used a shortened notation for the sum of all second-order terms. Let us study the properties of the terms $I_0,\,I_1$ and $I_2$. The first of them is nonpositive because of the condition 1$^\circ$. The expectation of $I_1$ is zero, by the properties of stochastic integrals. To handle the last term, note that by a simple differentiation, 3$^\circ$ implies
$$ \big\langle D^2B(x,y,w,v)(d,e,r,s),(d,e,r,s)\big\rangle\leq 0$$
for any $(x,y,w,v)\in \mathbb{R}^2\times \mathcal{D}_c$ and any $(d,e,r,s)\in \R^4$ satisfying $|e|\leq |d|$. 
This implies $I_2\leq 0$, by a straightforward approximation of the integral by Riemann sums. Putting all the above observations together, we get $ \E B(Z_t)\leq 0,$ which combined with the majorization condition 2$^\circ$ gives the assertion.
\end{proof}

We conclude this section with three observations.

\begin{remark}
The above statement is true without the assumption that $B$ is of class $C^2$: it is enough to ensure that $B$ is continuous. Indeed, the condition 3$^\circ$ guarantees that any possible `cusp' of $B$ is of concave type and hence the argument works. More precisely, this can be proved by standard mollification argument (consult e.g. Domelevo and Petermichl \cite{DP} or Wang \cite{W}). There are also several other methods of showing this. One can use the appropriate extension of It\^o's formula developed in \cite{P}; alternatively, one can first establish the corresponding estimate for (discrete-time) martingales and use approximation: see \cite{Bu1} for details.
\end{remark}

\begin{remark}
The above approach works also in the unweighted case, which corresponds to the choice $c=1$. Then the processes $W$ and $V$ are constant, and hence the special function $B$ depends only on the variables $x$, $y$. This brings us back to the original setting considered by Burkholder.
\end{remark}

\begin{remark}\label{extension}
The above approach is very flexible and can be easily modified to other contexts. For example, suppose that we are interested in the maximal bound of the form
$$ \E G(X_t,Y_t,Y_t^*,W_t,V_t)\leq 0,\qquad t\geq 0,$$
for all $X$, $Y$, $W$ and $V$ as in \eqref{wewant}. Here $Y_t^*=\max_{0\leq s\leq t}Y_t$ is the truncated one-sided maximal function of $Y$. Then it is enough to construct $B:\{(x,y,z,w,v)\in \R^3\times \mathcal{D}_c:y\leq z\}\to \R$ satisfying
\begin{itemize}
\item[1$^\circ$] (Initial condition) We have $B(x,y,y,{w},{v}) \leq 0$ if $|y|\leq |x|$ and $1\leq \mathtt{wv}^{p-1} \leq c$.
\item[2$^\circ$] (Majorization property) We have $B\geq G$.
\item[3$^\circ$] (Concavity-type property) For any $(x,y,z,{w},{v})\in \R^3\times \mathcal{D}_{c}$ and $d,e,r,s \in \mathbb{R}$ satisfying $y<z$ and $|e|\leq |d|$, the function
\begin{equation*}
\xi_{B}(t) := B(x+td,y+te,z,{w}+tr,{v}+ts),
\end{equation*}
given for those $t$, for which $1 \leq ({w}+tr)({v}+ts) \leq c$, is locally concave. Furthermore, we have $ B_z(x,y,y,w,v)\leq 0.$
\end{itemize}
Again, the proof rests on It\^o's formula. The additional requirement formulated at the end of 3$^\circ$ enables us to handle the additional stochastic integral $\int_{0+}^t B_z(X_s,Y_s,Y_s^*,W_s,V_s)\mbox{d}X_s^*$ (and guarantees that this integral is nonpositive).
\end{remark}

\section{A special function}

Throughout this section, $c>1$ is a fixed parameter (which corresponds to the `truly' weighted context). Again, as discussed in the paragraph following \eqref{L2DP}, we only consider the case $p=2$.  The main result of this section is the following.

\begin{theorem}
There is a continuous function $B:\mathbb{R}^2\times \mathcal{D}_{c} \to \mathbb{R}$ satisfying 1$^\circ$-3$^\circ$ with $G(x,y,w,v)=\kappa (y^{2}{w}-C^{2}c^{2}x^{2}{v}^{-1})$ for some positive universal constants $\kappa$ and $C$.
\end{theorem}

The above statement combined with Lemma \ref{MainLemma} yields the validity of \eqref{L2DP}, by passing $t\to \infty$ and using standard limiting arguments. A slightly stronger, maximal estimate announced in Introduction will be proved at the end of this section. 
Assume the $D^{1}, D^{2}, D^{3}$ are the `angular' subsets of $\R^2\times \mathcal{D}_{c}$ given by
\begin{align}\label{Domains}
D^1 &=\left\{(x,y,{w},{v}): |y| \geq 20c|x|(c/t)^{1-\beta}  \right\}, \nonumber \\
D^2 &=\left\{(x,y,{w},{v}): 10|x|\leq|y| \leq 20c|x|(c/t)^{1-\beta}  \right\}, \\
D^3 &=\left\{(x,y,{w},{v}): |y|\leq 10|x|  \right\} \nonumber.
\end{align}
Here and in what follows, we denote $t={w}{v}$. 
Define the functions $b_{i}:\mathcal{D}_c\to \R$ by
\begin{align*}
b_{1}(x,y,{w},{v}) &= y^{2}{w}\phi({w}{v}), \\
b_{2}(x,y,{w},{v}) &= y^{2}(2{v})^{-1},\\
b_{3}(x,y,{w},{v}) &= c^2x^{2}{v}^{-1},\\
b_{4}(x,y,{w},{v}) &= c^{\beta}|x||y|{w}^{1-\beta} {v}^{-\beta},\\
b_{5}(x,y,{w},{v}) &= c^{\beta}y^{2}{w}^{1-\beta}{v}^{-\beta},\\
b_{6}(x,y,{w},{v}) &= c^{2}x^{2}{w}\psi({w}{v}),
\end{align*}
where $\beta=3/4$ and $\phi,\psi$ are functions from $[1,c]$ to $\mathbb{R}$ given by
\begin{align*}
\phi(t) &= 2 - \frac{1}{t} - \frac{\ln(t)}{2c},\qquad \psi(t) = (t\phi(t))^{-1}.
\end{align*}
Furthermore, set $U(x,y,{w},{v})=b_1-b_2-320000b_3-294400b_6$. Now we are finally ready to introduce the explicit formula for the desired Burkholder's function $B$:
$$ B(x,y,{w},{v})=\begin{cases}
B_1(x,y,{w},{v}) & \mbox{on }D^1, \\
B_2(x,y,{w},{v}) &\mbox{on }D^2, \\
B_3(x,y,{w},{v}) & \mbox{on }D^3,
\end{cases}$$
where $B_1$, $B_2$, $B_3:\R^2\times \mathcal{D}_c\to \R$ are given by
\begin{align*}
 B_1(x,y,{w},{v})&=U(x,y,{w},{v})+6400c^{2}x^{2}{v}^{-1},\\
B_2(x,y,{w},{v})&=U(x,y,{w},{v})+320b_{4}(x,y,{w},{v}),\\
B_3(x,y,{w},{v})&=U(x,y,{w},{v})+32b_{5}(x,y,{w},{v}).
\end{align*}

As we have already announced earlier, this function has quite a complicated form, it is actually defined with three different formulas on three separate domains. Let us verify that it satisfies the conditions 1$^\circ$-3$^\circ$ listed in the formulation of Lemma \ref{MainLemma}. The first two properties are relatively easy to prove; the main difficulty lies in establishing the concavity condition. We start with the easy part. Note that the second half of the lemma below implies the continuity of $B$.

\begin{lemma} The function $B$ satisfies the properties 1$^{\circ}$ and 2$^{\circ}$. Furthermore, we have $B_{1}\leq B_{2}$ on $D^1$, $B_{2} \leq \min(B_{1},B_{3})$ on $D^2$ and $B_{3}\leq B_{2}$ on $D^3$.
\end{lemma}
\begin{proof}
To check the initial condition, note that for $|y|\leq |x|$ we have $(x,y,{w},{v}) \in D^3$. Furthermore, from $\phi(t) \leq 2$, we obtain
\begin{align*}
B_{3}(x,y,{w},{v})  \leq b_{1}+32b_{5}-320000b_{3} &\leq x^{2}{w}(2+32c^{\beta}({w}{v})^{-\beta}-320000(c^{2}/t))\\
& \leq x^{2}{w}(2+32(c/t)^{\beta}-320000(c/t)c) \leq 0.
\end{align*}
Let us now study the majorization. Observe that
\begin{align*}
b_{1}-b_{2}=y^{2}{w}\left[2-\frac{1}{t}-\frac{\ln t}{2c}-\frac{1}{2t}\right] \geq \frac{1}{2}y^{2}{w},
\end{align*}
because the function in the square bracket is increasing and has its minimum at the point $t=1$. Now from the estimate $\phi(t) \geq 1$ we have that $\psi(t) \leq 1/t$ and as a consequence,
\begin{align*}
320000b_{3}+294400b_{6} \leq c^{2}x^{2}{w}\left(320000\frac{1}{t}+294400\frac{1}{t} \right) = 614400c^{2}x^{2}{v}^{-1}.
\end{align*}
Finally, we have
\begin{align*}
B \geq b_{1}-b_{2}-320000b_{3}-294400b_{6} &\geq \frac{1}{2}y^{2}{w}-614400c^{2}x^{2}{v}^{-1}\\
& = \frac{1}{2} (y^{2}{w}-1228800c^{2}x^{2}{v}^{-1}),
\end{align*}
so the condition 2$^{\circ}$ is satisfied with $\kappa=1/2$ and $C=(1228800)^{1/2}<1109$. 

It remains to verify the relations between $B_1$, $B_2$ and $B_3$.  If $(x,y,{w},{v}) \in D^1$, then
\begin{align*}
320b_{4}=320c^{\beta}|x||y|{w}^{1-\beta}{v}^{-\beta} \geq 6400c^{2}x^{2}{v}^{-1},
\end{align*}
so $B_{2}\geq B_{1}$. If $(x,y,{w},{v}) \in D^2$, we have reverse inequality $B_{2} \leq B_{1}$. Furthermore, on $D^2$ we have
\begin{align*}
320b_{4}=320c^{\beta}|x||y|{w}^{1-\beta}{v}^{-\beta} \leq 32 c^{\beta}|y|^{2}{w}^{1-\beta}{v}^{-\beta} = 32b_{5},
\end{align*}
which is exactly $B_2\leq B_{3}$. To finish the proof, observe that the above estimate is reversed on $D^3$.
\end{proof}

We turn our attention to the crucial condition 3$^{\circ}$. From symmetry (and the equality $B_x(0,y,{w},{v})=0$ for all $y$, ${w}$, ${v}$), without loss of generality, we may only consider points $(x,y,{w},{v})\in \R^2\times \mathcal{D}_{c}$ such that $x$ and $y$ are nonnegative. Furthermore,  it is enough to verify the version of the concavity ``localized''  to each $D^i$. More precisely, it suffices to show that for each $i$, each $(x,y,{w},{v})\in D^i$ and any $d,e,r,s \in \mathbb{R}$ satisfying $|e|\leq |d|$, the function
\begin{equation*}
\xi_{B_i}(t) := B_i(x+td,y+te,{w}+tr,{v}+ts),
\end{equation*}
given for those $t$, for which $(x+td,t+te,{w}+tr,{v}+ts)\in D^i$, satisfies $\xi_{B_i}''(0)\leq 0$. To see that this is sufficient,   suppose that we have successfully established the localized concavity and pick an arbitrary point $(x,y,w,v)$ from the domain of $B$. By the continuity of $B$, we may and do assume that $x$ and $y$ are not both $0$. If $(x,y,w,v)$ belongs to the interior of some $D^i$, then $\xi_B(t)=\xi_{B^i}(t)$ for $t$ sufficiently close to $0$ and hence $\xi_B''(0)=\xi_{B_i}''(0)\leq 0$. On the other hand, if $(x,y,w,v)$ lies on the common boundary of two sets $D^i$ and $D^j$, then by the second part of the above lemma, 
$$\xi_B(t)=\min\{\xi_{B_i}(t),\xi_{B_j}(t)\}=\begin{cases}
\xi_{B^i}(t) &\mbox{if }(x+td,t+te,{w}+tr,{v}+ts)\in D^i,\\
\xi_{B^j}(t) &\mbox{if }(x+td,t+te,{w}+tr,{v}+ts)\in D^j
\end{cases}$$
 for $t$ sufficiently close to $0$. Hence, if $\xi_B$ had a convex ``cusp'' at zero, then the same would be true for $\xi_{B^i}$ and $\xi_{B^j}$, which contradicts the localized concavity. This establishes the desired property 3$^\circ$.

The localized concavity will be accomplished by a careful analysis of the derivatives $\xi_{b_j}''(0)$ of the building blocks $b_j$, $j=1,\,2,\,\ldots,\,6$. In the next lemma we gather estimates for the parts $b_{1}$ and $b_{6}$.
\begin{lemma} \label{PropertiesQ}
We have the following estimates on the quadratic forms associated with the functions $b_{1}$ and $b_{6}$:
\begin{itemize}
\item[(a)] $\xi''_{b_{1}}(0) \leq 80c{w}e^{2}$,
\item[(b)] $\xi''_{b_{1}}(0) \leq 4{w}e^{2}+8y|e||r|$,
\item[(c)] $\xi''_{b_{6}}(0) \geq (1/16)c{v}^{-3}x^{2}s^{2}$.
\end{itemize}
\end{lemma}
\begin{proof}
(a) It is equivalent to showing the nonpositive-definiteness of the matrix
\begin{align*}
\mathbb{A}(y,{w},{v})\!=\!\begin{pmatrix}
  2{w}\phi(t)-80c{w} & 2y\phi(t)+2yt\phi'(t) & 2y{w}^{2}\phi'(t) \\
  2y\phi(t)+2yt\phi'(t) & 2y^{2}{v}\phi'(t)+y^{2}t{v}\phi''(t) & 2y^{2}{w}\phi'(t)+y^{2}t{w}\phi''(t)\\
  2y{w}^{2}\phi'(t) & 2y^{2}{w}\phi'(t)+ y^{2}t{w}\phi''(t) & y^{2}{w}^{3}\phi''(t) 
 \end{pmatrix}.
\end{align*}
From Sylvester's criterion, it is enough to prove that
\begin{align}\label{CondA1}
y^{2}{w}^{3}\phi''(t) \leq 0,
\end{align}
\begin{align}\label{CondA2}
\det\begin{pmatrix}
  2y^{2}{v}\phi'(t)+y^{2}t{v}\phi''(t) & 2y^{2}{w}\phi'(t)+y^{2}t{w}\phi''(t)\\
  2y^{2}{w}\phi'(t)+ y^{2}t{w}\phi''(t) & y^{2}{w}^{3}\phi''(t) 
 \end{pmatrix} \geq 0
\end{align}
and
\begin{align}\label{CondA3}
\det\mathbb{A}(y,{w},{v}) \leq 0.
\end{align}
The inequality (\ref{CondA1}) follows immediately from $t \in [1,c]$ and the estimate
\begin{align*}
y^{2}{w}^{3}\phi''(t)=-\frac{y^{2}{w}^{3}}{2ct^{3}}(4c-t) \leq 0.
\end{align*}
The inequality (\ref{CondA2}) is equivalent to $\phi'(t)(2\phi'(t)+t\phi''(t))\leq 0$ and follows from 
\begin{align*}
\phi'(t)=\frac{1}{2ct^{2}}(2c-t) \geq 0 \qquad \mbox{ and }\qquad 2\phi'(t)+t\phi''(t)=-\frac{1}{2ct} \leq 0.
\end{align*}
In order to show (\ref{CondA3}) we simplify the matrix $\mathbb{A}$ by carrying out some elementary operations. The determinant of $\mathbb{A}$ has the same sign as
\begin{align*}
& \det\begin{pmatrix}
  -80c{w} & 2\phi(t)+2t\phi'(t) & 0 \\
  2\phi(t) & 0 & 2\phi'(t)\\
  2{w}\phi'(t) & 2\phi'(t)+ t\phi''(t) & {w}\phi''(t) 
 \end{pmatrix} \\
 & = 4{w} \left[\left(2(\phi'(t))^{2}-\phi(t)\phi''(t)\right)(\phi(t)+t\phi'(t))+40c\phi'(t)(2\phi'(t)+t\phi''(t))\right].
\end{align*}
We compute that
\begin{align} \label{PhiIneq}
\phi(t)+t\phi'(t)=2-\frac{\ln(t)}{2c}-\frac{1}{2c} \leq 2,
\end{align}
\begin{align*}
2(\phi'(t))^{2}=\phi'(t)\frac{2c-t}{ct^{2}}\leq\frac{2\phi'(t)}{t}
\end{align*}
and, since $\phi(t)\leq 2$,
\begin{align}\label{PhiIneq2}
-\frac{\phi(t)\phi''(t)}{\phi'(t)}\leq \frac{2\left(\frac{2}{t^{3}}-\frac{1}{2ct^{2}}\right)}{\frac{1}{t^{2}}-\frac{1}{2ct}} \leq \frac{8}{t}.
\end{align}
Combining these facts we obtain
\begin{align*}
(2(\phi'(t))^{2}-\phi(t)\phi''(t))(\phi(t)+t\phi'(t))\leq \frac{20\phi'(t)}{t},
\end{align*}
and since
\begin{align*}
40c\phi'(t)(2\phi'(t)+t\phi''(t))=-\frac{20\phi'(t)}{t},
\end{align*}
the inequality (\ref{CondA3}) is satisfied. This completes the proof of the part $(a)$.

(b) Firstly, observe that it is sufficient to prove the nonpositive-definiteness of the matrix
\begin{align*}
\mathbb{B}(y,{w},{v})=\begin{pmatrix}
  2{w}\phi(t)-4{w} & 0 & 2y{w}^{2}\phi'(t) \\
  0 & 2y^{2}{v}\phi'(t)+y^{2}t{v}\phi''(t) & 2y^{2}{w}\phi'(t)+y^{2}t{w}\phi''(t)\\
  2y{w}^{2}\phi'(t) & 2y^{2}{w}\phi'(t)+ y^{2}t{w}\phi''(t) & y^{2}{w}^{3}\phi''(t) 
 \end{pmatrix}.
\end{align*}
Indeed, we have the estimate
\begin{align*}
\xi''_{b_1}(0) &=\left \langle\mathbb{B}(y,{w},{v})(e,r,s),(e,r,s) \right \rangle + 4{w}e^{2}+2(2y\phi(t)+2yt\phi'(t))er \\
&\leq 4{w}e^{2}+4(y\phi(t)+yt\phi'(t))|e||r| \leq 4{w}e^{2}+8y|e||r|,
\end{align*}
where the last inequality follows from (\ref{PhiIneq}).

From Sylvester's criterion, the nonpositive-definiteness of the matrix $\mathbb{B}$ is equivalent to inequalities (\ref{CondA1}) and (\ref{CondA2}) (which we already showed in the proof of the $(a)$ part of the lemma) and the estimate
\begin{align}\label{CondB3}
\det \mathbb{B}(y,{w},{v}) \leq 0.
\end{align}
By carrying out some elementary operations we show that the determinant of $\mathbb{B}$ has the same sign as
\begin{align*}
&\det \begin{pmatrix}
  2{w}\phi(t)-4{w}+2{w}t\phi'(t) & 0 & 2{w}\phi'(t) \\
  0 & 0 & 2\phi'(t)\\
  2{w}\phi'(t)+{w}t\phi''(t) & 2\phi'(t)+ t\phi''(t) & {w}\phi''(t) 
 \end{pmatrix}\\
 &=-(2\phi'(t)+t\phi''(t))2\phi'(t)(2{w}\phi(t)-4{w}+2{w}t\phi'(t)).
\end{align*}
However, we compute that
\begin{align*}
2\phi'(t)+t\phi''(t)=-\frac{1}{2tc} \leq 0.
\end{align*}
So, since $\phi(t)+t\phi'(t) \leq 2$ and $\phi'(t) \geq 0$,
the inequality (\ref{CondB3}) is satisfied.

(c) In analogy to the above considerations, we must show that the matrix
\begin{align*}
&\mathcal{C}(x,{w},{v})\\
&=\begin{pmatrix}
  2c^{2}{w}\psi(t) & 2xc^{2}(\psi(t)+t\psi'(t)) & 2xc^{2}{w}^{2}\psi'(t) \\
  2xc^{2}(\psi+t\psi'(t)) & x^{2}c^{2}(2{v}\psi'(t)+{w}{v}^{2}\psi''(t)) & x^{2}c^{2}(2{w}\psi'(t)+{w}^{2}{v}\psi''(t))\\
  2xc^{2}{w}^{2}\psi'(t) & x^{2}c^{2}(2{w}\psi'(t)+{w}^{2}{v}\psi''(t)) & x^{2}c^{2}{w}^{3}\psi''(t)-\frac{1}{16}c{v}^{-3}x^{2} 
 \end{pmatrix}
\end{align*}
is nonpositive-definite. 
For notational convenience, let us define the function $\widehat{\psi}:[1,c]\longrightarrow \mathbb{R}$ as $\widehat{\psi}(t)=t\psi(t)=(\phi(t))^{-1}$ and set 
$$d(t,x)=x^{2}c^{2}(2{v}^{-3}\widehat{\psi}-2{v}^{-2}{w}\widehat{\psi}'+{v}^{-1}{w}^{2}\widehat{\psi}'')-\frac{1}{16}{v}^{-3}cx^{2}.$$
Then we can rewrite the matrix $\mathcal{C}$ as
\begin{align*}
\begin{pmatrix}
  2c^{2}\widehat{\psi}(t){v}^{-1} & 2c^{2}x\widehat{\psi}'(t) & 2xc^{2}(\widehat{\psi}'(t){w}{v}^{-1}-\widehat{\psi}(t){v}^{-2}) \\
  2c^{2}x\widehat{\psi}'(t) & c^{2}{v}x^{2}\widehat{\psi}''(t) & c^{2}{w}x^{2}\widehat{\psi}''(t)\\
  2xc^{2}(\widehat{\psi}'(t){w}{v}^{-1}-\widehat{\psi}(t){v}^{-2}) & c^{2}{w}x^{2}\widehat{\psi}''(t) & d(x,t)
 \end{pmatrix}.
\end{align*}
Again, from Sylvester's criterion, we reduce the problem to checking the signs of appropriate  minors. More precisely, we will show that
\begin{align}\label{CondC1}
2c^{2}\widehat{\psi}{v}^{-1} \geq 0,
\end{align}
\begin{align}\label{CondC2}
\det\begin{pmatrix}
  2c^{2}\widehat{\psi}(t){v}^{-1} & 2c^{2}x\widehat{\psi}'(t)\\
  2c^{2}x\widehat{\psi}'(t) & c^{2}{v}x^{2}\widehat{\psi}''(t)
 \end{pmatrix} \geq 0
\end{align}
and
\begin{align}\label{CondC3}
\det \mathcal{C}(x,{w},{v}) \geq 0.
\end{align}
The inequality (\ref{CondC1}) is obvious. Condition (\ref{CondC2}) is equivalent to $\widehat{\psi}(t)\widehat{\psi}''(t)-2(\widehat{\psi}'(t))^{2} \geq 0$, which is a consequence of the definition $\widehat{\psi}(t)=(\phi(t))^{-1}$ and the inequality $\phi''(t) \leq 0$. To show (\ref{CondC3}), we perform certain elementary operations on the columns and rows of the matrix to  prove that the determinant of $\mathcal{C}$ has the same sign as
\begin{align*}
&\det \begin{pmatrix}
  2\widehat{\psi}(t){v}^{-1} & 2\widehat{\psi}'(t) & 0 \\
  2\widehat{\psi}'(t) & {v}\widehat{\psi}''(t) & 2{v}^{-1}\widehat{\psi}'(t)\\
  -2\widehat{\psi}(t){v}^{-2} & 0 & -2w{v}^{-2}\widehat{\psi}'(t)-\frac{1}{16}{v}^{-3}c^{-1}
 \end{pmatrix} \\
&= 2{v}^{-3}\left(\left(-2\widehat{\psi}'(t)t-\frac{1}{16}c^{-1}\right)(\widehat{\psi}(t)\widehat{\psi}''(t)-2(\widehat{\psi}'(t))^{2})-4\widehat{\psi}(t)(\widehat{\psi}'(t))^{2} \right).
\end{align*}
We compute that
\begin{align*}
\widehat{\psi}'(t)=-\frac{\phi'(t)}{\phi^{2}(t)},
\end{align*}
\begin{align*}
\widehat{\psi}''(t)=-\frac{\phi(t)\phi''(t)-2(\phi'(t))^{2}}{\phi^{3}(t)}
\end{align*}
and
\begin{align*}
\widehat{\psi}(t)\widehat{\psi}''(t)-2(\widehat{\psi}'(t))^{2} = -\frac{\phi''(t)}{\phi^{3}(t)}.
\end{align*}
Hence we need to show that
\begin{align*}
-2\phi'(t)\phi''(t)t+\frac{1}{16}c^{-1}\phi''(t)(\phi(t))^{2}-4(\phi'(t))^{2}\geq 0.
\end{align*}
Now observe that
\begin{align*}
-2\phi'(t)\phi''(t)t - 4(\phi'(t))^{2}=-2\phi'(t)(2\phi'(t)+\phi''(t)t)=\phi'(t)c^{-1}t^{-1}
\end{align*}
and from (\ref{PhiIneq2}) and $\phi(t) \leq 2$
\begin{align*}
\frac{1}{16}c^{-1}\phi''(t)(\phi(t))^{2} \geq \frac{-\phi(t)\phi'(t)}{2ct} \geq -\frac{\phi'(t)}{ct},
\end{align*}
which completes the proof of the lemma.
\end{proof}

In the series of three lemmas below we will show that the function $B$ satisfies required concavity condition. Let us start with the domain $D^1$.
\begin{lemma}\label{Domain1}
We have $\xi''_{b_{1}-b_{2}-160b_{3}}(0) \leq 0$ for any $(x,y,{w},{v}) \in D^1$ and $(d,e,r,s)$ such that $|e|\leq |d|$.
\end{lemma}

\begin{rem} This lemma handles the property 3$^{\circ}$ on the domain $D^1$. Indeed, the additional summands $-319840b_{3}$ and $-294400b_{6}$ are concave functions (concavity of $-b_{3}$ is easy to check, concavity of $-b_{6}$ follows from part $(c)$ of Lemma \ref{PropertiesQ}).
\end{rem}
\begin{proof} [Proof of Lemma \ref{Domain1}] We have that
\begin{align*}
\xi''_{b_2}(0) &=\frac{1}{{v}}\left(e-\frac{ys}{{v}}\right)^{2}, \qquad \xi''_{b_3}(0) =\frac{2c^{2}}{{v}}\left(d-\frac{xs}{{v}}\right)^{2}.
\end{align*}
Now consider two cases. If $|d-\frac{xs}{{v}}| \geq d/2$, then from above formulas and part $(a)$ of Lemma \ref{PropertiesQ} we obtain
\begin{align*}
\xi''_{b_{1}-b_{2}-160b_{3}}(0) &\leq 80c{w}e^{2}-\frac{1}{{v}}\left(e-\frac{ys}{{v}}\right)^{2}-\frac{320c^{2}}{{v}}\left(d-\frac{xs}{{v}}\right)^{2} \\
&\leq 80c{w}e^{2}-80c^{2}{v}^{-1}d^{2} \\
&\leq 80c{w}d^{2}-80c{w}d^{2} \\
&= 0.
\end{align*}
If $|d-\frac{xs}{{v}}|<d/2$, then
\begin{align*}
\frac{ys}{d{v}}-\frac{e}{d} &=\frac{ys}{d{v}}-\frac{y}{x}+\frac{y}{x}-\frac{e}{d}= y\left(\frac{s}{d{v}}-\frac{1}{x} \right) + \frac{y}{x} - \frac{e}{d} = \frac{y}{x} \left( \frac{sx}{{v}d}-1+1-\frac{e}{d}\frac{x}{y} \right) \\
 &\geq \frac{y}{x} \left( -\frac{1}{2}+1-\frac{1}{20c} \right) \geq 20c \left(\frac{1}{2}-\frac{1}{20c} \right)=10c-1 \geq 9c.
\end{align*}
Hence
\begin{align*}
\xi''_{b_{1}-b_{2}-160b_{3}}(0) \leq \xi''_{b_1-b_2}(0) \leq 80c{w}e^{2}-\frac{1}{{v}}d^{2}9^{2}c^{2} \leq 80c{w}d^{2}-81c{w}d^{2} \leq 0.
\end{align*}
The proof is complete.
\end{proof}
The next lemma discusses the concavity condition in the middle domain $D^2$.
\begin{lemma}\label{Domain2}
We have $\xi''_{b_{1}+320b_{4}-320000b_{3}}(0) \leq 0$ for any $(x,y,{w},{v}) \in D^2$ and $(d,e,r,s)$ such that $|e|\leq |d|$.
\end{lemma}

\begin{rem}
 This lemma handles the property 3$^{\circ}$ on the domain $D^2$. Indeed, functions $-b_{2}$ and $-294400b_{6}$ are concave, so they do not affect the condition 3$^{\circ}$.
\end{rem}
\begin{proof} [Proof of Lemma \ref{Domain2}] Let $D=\frac{d}{x}$, $E=\frac{e}{y}$, $R=\frac{r}{{w}}$ and $S=\frac{s}{{v}}$. We compute that
\begin{align*}
\xi_{b_{4}-1000b_{3}}(0) &=c^{\beta}xy{w}^{1-\beta}{v}^{-\beta} \left \langle A_{1}(E,D,R,S),(E,D,R,S)\right \rangle \\
& - 1000c^{2}x^{2}{v}^{-1} \left \langle A_{2}(D,S),(D,S)\right \rangle,
\end{align*}
where the matrices $A_1$ and $A_{2}$ are defined as
\begin{align*}
A_{1}=\begin{pmatrix}
  0 & 1 & 1-\beta & - \beta \\
  1 & 0 & 1-\beta & - \beta \\
  1-\beta & 1- \beta & \beta(\beta-1) & \beta(\beta-1) \\
  -\beta & -\beta & \beta(\beta-1) & \beta(\beta+1)
 \end{pmatrix}
\end{align*}
and
\begin{align*}
A_{2}=\begin{pmatrix}
  2 & -2 \\
  -2 & 2 
 \end{pmatrix}.
\end{align*}
From the assumption $y \geq 10x$ and differential subordination ($|e| \leq |d|$) we obtain
\begin{align*}
|E|=\left|\frac{e}{y}\right|\leq \frac{1}{10}\left|\frac{d}{x}\right|=\frac{1}{10}|D|,
\end{align*}
so $E=\lambda D$, where $\lambda$ is a constant with absolute value bounded by $\frac{1}{10}$. So, we can reduce Hessians to three variables ($D$,$R$ and $S$): we have
\begin{align}\label{EqDomain2}
\xi''_{b_{4}-1000b_{3}}(0) &=c^{\beta}xy{w}^{1-\beta}{v}^{-\beta}\left \langle A_{3}(D,R,S),(D,R,S)\right \rangle \\
&\quad - 1000c^{2}x^{2}{v}^{-1} \left \langle A_{4}(D,R,S),(D,R,S)\right \rangle,  \nonumber
\end{align}
where matrices $A_{3}$ and $A_{4}$ are defined as
\begin{align*}
A_{3}=\begin{pmatrix}
  2\lambda & (1-\beta)(1+\lambda) & - \beta(1+\lambda) \\
  (1-\beta)(1+\lambda) & \beta(\beta-1) & \beta(\beta-1) \\
  - \beta(1+\lambda) & \beta(\beta-1) & \beta(\beta+1)
 \end{pmatrix}
\end{align*}
and
\begin{align*}
A_{4}=\begin{pmatrix}
  2 & 0 & -2 \\
  0 & 0 & 0 \\
  -2 & 0 & 2
 \end{pmatrix}.
\end{align*}
Now from $y \leq 20cx(c/t)^{1-\beta}$ we have the estimate
\begin{align*}
1000c^{2}x^{2}{v}^{-1} \geq 50xy{v}^{-1}c(c/t)^{\beta-1}=50c^{\beta}xy{w}^{1-\beta}{v}^{-\beta}.
\end{align*}
Obviously, $A_{4}$ is nonnegative-definite. Hence, from the above inequality and (\ref{EqDomain2}), we obtain
\begin{align*}
\xi''_{b_{4}-1000b_{3}}(0)\leq c^{\beta}xy{w}^{1-\beta}{v}^{-\beta} \left \langle (A_{3}-50A_{4})(D,R,S),(D,R,S)\right \rangle.
\end{align*}
It is enough to show that
\begin{align}\label{Eq2Domain2}
\left \langle (A_{3}-50A_{4})(D,R,S),(D,R,S) \right \rangle \leq -\frac{1}{20}D^{2}-\frac{1}{40}|D||R|.
\end{align}
Indeed, from the above inequalities and part $(b)$ of Lemma \ref{PropertiesQ} we have that
\begin{align*}
\xi''_{b_{1}+320(b_{4}-1000b_{3})}(0) &\leq 4{w}e^{2}+8y|e||r|-16c^{\beta}xy{w}^{1-\beta}{v}^{-\beta}D^{2}-8c^{\beta}xy{w}^{1-\beta}{v}^{-\beta}|D||R| \\
& \leq 4{w}d^{2}+8y|d||r|-16{w}d^{2}(c/t)^{\beta}(y/x)-8y|d||r|(c/t)^{\beta}\\
& \leq 0.
\end{align*}
To finish the proof of the lemma, observe that the estimate (\ref{Eq2Domain2}) is equivalent to nonpositive-definiteness of the matrix
\begin{align*}
\begin{pmatrix}
  2\lambda-100+\frac{1}{20} & \frac{1}{4}(1+\lambda) \pm \frac{1}{40} & - \frac{3}{4}(1+\lambda)+100 \\
  \frac{1}{4}(1+\lambda) \pm \frac{1}{40} & -\frac{3}{16} & - \frac{3}{16} \\
  - \frac{3}{4}(1+\lambda)+100 & - \frac{3}{16} & \frac{21}{16}-100
\end{pmatrix},
\end{align*}
for every $|\lambda|\leq 1/10$, which we check by straightforward calculation (determinant of this matrix is convex as a function of $\lambda$, so it is sufficient to check only two endpoint cases $\lambda=1/10$ and $\lambda=-1/10$).
\end{proof}
Finally, we prove the concavity condition in the domain $D^3$ in the last lemma.

\begin{lemma}\label{Domain3}
We have $\xi''_{b_{1}+32b_{5}-4600b_{3}-294400b_{6}}(0) \leq 0$ for any $(x,y,{w},{v}) \in D^3$ and $(d,e,r,s)$ such that $|e|\leq |d|$.
\end{lemma}

\begin{rem} This lemma handles the property 3$^{\circ}$ on the domain $D^3$. Indeed, the additional summands $-b_{2}$ and $-315400b_{3}$ are concave, so they do not affect the concavity.
\end{rem}
\begin{proof} [Proof of Lemma \ref{Domain3}]
We use the same notation for relative changes $D, E, R$ and $S$ as in the proof of the previous lemma. We start with the analysis of the part $b_{3}$. We have that
\begin{align*}
\xi''_{b_{3}}(0)=\frac{2c^2}{{v}}\left( d - \frac{xs}{{v}} \right)^{2}=\frac{2c^{2}x^{2}}{{v}}(D-S)^{2} \geq 
\frac{2cx^{2}}{{v}}(D-S)^{2} = \frac{2cx^{2}{w}}{t}(D-S)^{2}.
\end{align*}
From the part $(c)$ of Lemma \ref{PropertiesQ} and the above estimate we obtain
\begin{align*}
\xi''_{b_{3}+64b_{6}}(0) &\geq 2 \left(\frac{c}{t}x^{2}{w}(D^{2}-2DS+S^{2})+2c{v}^{-3}x^{2}s^{2} \right) \\
&= 
2 \left( \frac{c}{t}x^{2}{w}(D^{2}-2DS+S^{2}) + 2\left(\frac{c}{t} \right)x^{2}{w}S^{2} \right) \\
&=
2\left(\frac{c}{t}\right) x^{2}{w}(D^{2}-2DS+3S^{2}) \\
& \geq
\left(\frac{c}{t}\right)x^{2}{w}(D^{2}+2S^{2}) \\
& \geq \left( \frac{c}{t} \right)y^{2}{w}\left(E^{2}+\frac{2}{100}S^{2} \right),
\end{align*}
hence
\begin{align} \label{PartialIneq}
\xi''_{16^{-1}\cdot23\cdot100(b_{3}+64b_{6})}(0) \geq \left(\frac{c}{t}\right)y^{2}{w}16^{-1}[2300E^{2}+46S^{2}].
\end{align}
Now we turn our attention to the analysis of the part $b_{5}$. We have that
\begin{align*}
\xi''_{b_{5}}(0)=\left(\frac{c}{t}\right)^{\beta}y^{2}{w} \left \langle A_{5}(E,R,S),(E,R,S) \right \rangle,
\end{align*}
where
\begin{align*}
A_{5} = \begin{pmatrix}
  2 & 2(1-\beta) & - 2\beta \\
  2(1-\beta) & \beta(\beta-1) & \beta(\beta-1) \\
  - 2\beta & \beta(\beta-1) & \beta(\beta+1)
\end{pmatrix}.
\end{align*}
We check by straightforward calculation that

\begin{align*}
\begin{pmatrix}
  2 & 2(1-\beta) & - 2\beta \\
  2(1-\beta) & \beta(\beta-1) & \beta(\beta-1) \\
  - 2\beta & \beta(\beta-1) & \beta(\beta+1)
\end{pmatrix} \leq
16^{-1}
\begin{pmatrix}
  192 & \pm 2 & 0 \\
  \pm 2 & 0 & 0 \\
  0 & 0 & 46
\end{pmatrix}.
\end{align*}
So
\begin{align*}
\xi''_{b_{5}}(0) \leq \left(\frac{c}{t}\right)^{\beta}y^{2}{w} 16^{-1}(192E^{2}-4|E||R|+46S^{2})
\end{align*}
and, from (\ref{PartialIneq}),
\begin{align*}
\xi''_{b_{5}-16^{-1}\cdot23\cdot100(b_{3}+64b_{6})}(0) &\leq \left(\frac{c}{t}\right)^{\beta}y^{2}{w} 16^{-1}(-2108E^{2}-4|E||R|) \\
&\leq y^{2}{w}16^{-1}(-2108E^{2}-4|E||R|).
\end{align*}
Now, from the above estimate an part $(b)$ of Lemma \ref{PropertiesQ}, we have that
\begin{align*}
\xi''_{b_{1}+32b_{5}-4600b_{3}-294400b_{6}} \leq -4212{w}e^{2}\leq0,
\end{align*}
which concludes the proof.
\end{proof}

\def\irrelevant{
The final lemma is a discrete-time concavity condition on $B$.

\begin{lemma}
Consider an arbitrary line segment with endpoints $\mathtt{P}=(x,y,{w},{v})$, $\mathtt{Q}=(x+h,y+k,{w}+r,{v}+s)$, $|k|\leq |h|$, which is entirely contained in $D_c$. For such a segment, we have
\begin{equation}\label{conv}
 B(\mathtt{Q})\leq B(\mathtt{P})+B_x(\mathtt{P})(Q_x-P_x)+B_y(\mathtt{P})(Q_y-P_y)+B_{w}(\mathtt{P})(Q_{w}-P_{w})+
B_{v}(\mathtt{P})(Q_{v}-P_{v}).
\end{equation}
Here if $\mathtt{P}\in D^1\cap D^2$, the partial derivatives are understood to be the appropriate derivatives of $B_1$; if $\mathtt{P}\in D^2\cap D^3$, the derivatives come from $B_2$. 
\end{lemma}
\begin{proof}
The condition 3$^\circ$ implies that the function $G(t)=B(x+th,y+tk,{w}+tr,{v}+ts)$ is concave on $[0,1]$. Consequently, we have $G(1)\leq G(0)+G'(0+)$, and this immediately gives \eqref{conv}.
\end{proof}

\section{Proof of \eqref{MainTheoremL2}}

We will start with the following useful interpretation of $A_{2}$ weights. Let $w$ be such a weight and suppose that $c= [w]_{A_2}$. Let $v=(v_n)_{n \geq 0}$ be the martingale given by $v_n=\mathbb{E}(w^{-1} | \mathcal{F}_{n})_{n\geq 0}$. It follows directly from Jensen's inequality that $w_{n}v_n \geq 1$ almost surely for all $n \geq 0$. Condition $A_{2}$ is equivalent to the reverse bound
\begin{equation*}
w_{n}v_{n} \leq c,  \qquad \text{ with probability 1}.
\end{equation*}
In other words, an $A_{2}$ weight of characteristic equal to $c$ gives rise to a two-dimensional martingale $(w,v)$ taking values in the domain
\begin{equation*}
D_c=\left\{(\texttt{w},\texttt{v}) \in (0, \infty) \times (0,\infty): 1\leq \texttt{wv} \leq c \right\}
\end{equation*}
and terminating at the lower boundary of this set. 

We will need the following geometric lemma.

\begin{lemma}\label{geomlemma}
Assume that $c>1$, $\alpha\in [2^{-d},1-2^{-d}]$ and suppose that points $\mathtt{P}$, $\mathtt{Q}$ and $\mathtt{R}=\alpha \mathtt{P}+(1-\alpha)\mathtt{Q}$ lie in $D_c$. Then the whole line segment $\mathtt{P}\mathtt{Q}$ is contained within $D_{2^dc}$.
\end{lemma}
\begin{proof}
Using a simple geometrical argument, it is enough to consider the case when the points  $\mathtt{P}$ and $\mathtt{R}$ lie on the curve $\texttt{wv}=c$ (the upper boundary of $D_c$) and $\mathtt{Q}$ lies on the curve $\texttt{wv}=1$ (the lower boundary of $D_c$). Then the line segment $\mathtt{R}\mathtt{Q}$ is contained within $D_c$, and hence also within $D_{2c}$, so  it is enough to ensure that the segment $\mathtt{P}\mathtt{R}$ is contained in $D_{2c}$. Let $\mathtt{P}=(\mathtt{P}_{w},\mathtt{P}_{v})$, $\mathtt{Q}=(\mathtt{Q}_{w},\mathtt{Q}_{v})$ and $\mathtt{R}=(\mathtt{R}_{w},\mathtt{R}_{v})$. We consider two cases. If $\mathtt{P}_{w}<\mathtt{R}_{w}$, then 
$$ \mathtt{P}_{v}=\alpha^{-1}\mathtt{R}_{v}-(\alpha^{-1}-1)\mathtt{Q}_{v}<\alpha^{-1}\mathtt{R}_{v}\leq 2^d\mathtt{R}_v,$$
so the segment $\mathtt{P}\mathtt{R}$ is contained in the quadrant $\{({w},{v}):{w}\leq \mathtt{R}_{w},\,{v}\leq 2^d\mathtt{R}_{v}\}$. Consequently, $\mathtt{P}\mathtt{R}$ lies below the hyperbola $\mathtt{wv}=2^dc$ passing through $(\mathtt{R}_{w},2^d\mathtt{R}_{v})$. This proves the assertion in the case $\mathtt{P}_{w}<\mathtt{R}_{w}$. In the case $\mathtt{P}_{w}\geq\mathtt{R}_{w}$ the reasoning is similar. Indeed, we check easily that the line segment $\mathtt{P}\mathtt{R}$ lies below the hyperbola $\mathtt{wv}=2^dc$ passing through $(2^d\mathtt{R}_{w},\mathtt{R}_{v})$.
\end{proof}

Now we will exploit the properties of $B$ to obtain the proof of \eqref{MainTheoremL2}.

\begin{lemma}
We have the weighted $L^2$ bound
\begin{equation}\label{L2}
\|g\|_{L^2(w)}\leq 1109\cdot 2^d[w]_{A_2}\|f\|_{L^2(w)}.
\end{equation}
\end{lemma}
\begin{proof}
Let $w$ be an $A_2$ weight and set $c=2^d[w]_{A_2}$. Furthermore, let $f$, $g$ be two martingales such that $g$ is differentially subordinate to $f$; we may assume that $\|f\|_{L^2(w)}<\infty$, since otherwise there is nothing to prove. Let $B=B_c$ be the function constructed in the previous section. For a fixed $n\geq 0$, we apply \eqref{conv} with $x=f_n$, $y=g_n$, ${w}=w_n$, ${v}=v_n$ and $h=df_{n+1}$, $k=dg_{n+1}$, $r=dw_{n+1}$ and $s=dv_{n+1}$. The application is allowed: indeed, we have $|dg_{n+1}|\leq |df_{n+1}|$ (by the differential subordination), furthermore, the points $(x,y,{w},{v})$ and $(x+h,y+k,{w}+r,{v}+s)$ belong to $\mathcal{D}_{c/2^d}$ and hence the interval which joins them is entirely contained in $\mathcal{D}_c$, by Lemma \ref{geomlemma}. To explain the latter statement, fix an atom $A$ of $\F_n$: the random variable $(w_n,v_n)$ is constant on $A$, denote its value by $\mathtt{P}$. Let $A_1$, $A_2$, $\ldots$, $A_k$ be the collection of all atoms of $\F_{n+1}$ into which the event $A$ is split. For any fixed $j\in \{1,2,\ldots,k\}$, the random variable $(w_{n+1},v_{n+1})$ is constant on $A_j$; denote its value by $\mathtt{Q}$. The crucial observation is that for any $i\neq j$, the random variable $(w_{n+1},v_{n+1})$ restricted to $A_i$ belongs to the convex set $\{(\texttt{w},\texttt{v})\in (0,\infty)^2: \texttt{w}\texttt{v}\geq 1\}$, and hence so does the average
$$ \mathtt{R}=(\mathtt{R}_{w},\mathtt{R}_{v}):=\frac{1}{\mathbb{P}(\Omega\setminus A_j)}\int_{\Omega\setminus A_j} (w_{n+1},v_{n+1})\mbox{d}\mathbb{P}.$$
We know that $\mathtt{P}$ and $\mathtt{Q}$ belong to $D_{c/2^d}$ and we need to show that $\mathtt{P}\mathtt{Q}$ is entirely contained in $D_c$. If $\mathtt{R}\in D_{c/2^d}$, then this follows from Lemma \ref{geomlemma}; the appropriate bound on the number $\alpha$ appearing in the assumption is due to the regularity of the filtration. However, if $\mathtt{R}\notin D_{c/2^d}$, then necessarily $\mathtt{R}_{w}\mathtt{R}_{v}>c/2^d$ (we know that $\mathtt{R}_{w}\mathtt{R}_{v}< 1$ cannot hold) and hence the whole line segment $\mathtt{P}\mathtt{Q}$ must be contained in $D_{c/2^d}$: otherwise, we would have $\mathtt{P}\notin D_{c/2^d}$, by the convexity of the set $\{(\texttt{w},\texttt{v})\in (0,\infty)^2: \texttt{w}\texttt{v}>c/2^d\}$.

Thus we have shown that the use of \eqref{conv} is permitted, and as the result we get
$$ B(f_{n+1},g_{n+1},w_{n+1},v_{n+1})\leq B(f_n,g_n,w_n,v_n)+I,$$
where $\mathbb{E}(I|\F_n)=0$. Indeed, the term $I$ is a linear combination of $df_{n+1}$, $dg_{n+1}$, $dw_{n+1}$ and $dv_{n+1}$ with $\F_n$-measurable coefficients. Consequently, we obtain that the sequence $(\E B(f_n,g_n,w_n,v_n))_{n\geq 0}$ is nonincreasing, so the majorization and the initial condition give
$$ \frac{1}{2}\left(\E g_n^2w_n-1109^2c^2\E f_n^2v_n^{-1}\right)\leq \E B(f_n,g_n,w_n,v_n)\leq \E B(f_0,g_0,w_0,v_0)\leq 0.$$
The function $(x,v)\mapsto x^2v^{-1}$ is convex, so $\E f_n^2 v_n^{-1}\leq \E f^2v^{-1}=\E f^2w$. Furthermore, we have $w_n=\E(w|\F_n)$, sothe above inequality implies
$$ \E g_n^2 w\leq 1109^2c^2 \E f^2w.$$
It remains to let $n\to \infty$ and apply Fatou's lemma to get the claim.
\end{proof}

\begin{lemma} We have the weighted maximal $L^2$-bound
\begin{equation}\label{L2max}
 \|g^*-g\|_{L^2(w)}\leq 1109\cdot 2^d[w]_{A_2}\|f\|_{L^2(w)}.
\end{equation}
\end{lemma}

%(\textcolor{red}{I wonder if the following inequality could hold.  Let  $M$ the Hardy-Litlewood max function and $T$ a singular integral, even just the Hilbert transform on $\R$,  $\|M(Tf)-Tf\|_L^2(w)\leq C[w]_{A_2}\|f\|_{L^2(w)}$ where $M$ is the maximal function and $T$ is a singular integral?  This some what reminds me of the Cotlar Lemma}) 
\begin{proof} The argumentation splits naturally into two parts.

\smallskip

\emph{Step 1. An auxiliary special function and its properties.} 
Introduce the function $U:\R^3\times D_{c}\to \R$ of five variables given by
$$ U(x,y,z,{w},{v})=B(x,y-(y\vee z),{w},{v}).$$ 
Consider an arbitrary line segment with endpoints $(x,y,{w},{v})$, $(x+h,y+k,{w}+r,{v}+s)$, which is entirely contained in $D_c$. We will prove that for such a segment, if $y\leq z$ and $\mathtt{P}=(x,y,z,{w},{v})$ and $\mathtt{Q}=(x+h,y+k, z,{w}+r,{v}+s)$, then
\begin{equation}\label{conv2}
 U(\mathtt{Q})\leq U(\mathtt{P})+U_x(\mathtt{P})(Q_x-P_x)+U_y(\mathtt{P})(Q_y-P_y)+U_{w}(\mathtt{P})(Q_{w}-P_{w})+
U_{v}(\mathtt{P})(Q_{v}-P_{v}).
\end{equation}
(The problematic points are handled with as in \eqref{conv}: if $\mathtt{P}$ is such that $(x,y-z,{w},{v})\in D^1\cap D^2$, we compute the appropriate derivatives with the use of the formula for $B$ on the subdomain $D^1$; etc). To show the above inequality, consider the function $ G(t)= U(x+th,y+tk,z,{w}+tr,{v}+ts)$; it is enough to prove that it is concave on the interval $[0,1]$, since this implies $G(1)\leq G(0)+G'(0+)$, which is precisely \eqref{conv2}. If $k=0$, the concavity of $G$ follows at once from 4$^\circ$, so we may assume that $k\neq 0$. By the definition of $U$,
$$ G(t)=\begin{cases}
 B(x+th,y+tk-z,{w}+tr,{v}+ts) & \mbox{if }y+tk\leq z,\\
 B(x+th,0,{w}+tr,{v}+ts) & \mbox{if }y+tk\geq  z
\end{cases}$$
and directly from 4$^\circ$, the function $G$ is concave on each of the intervals $\{t\in [0,1]:y+tk\leq z\}$ and $\{t\in [0,1]:y+tk\geq  z\}$. Furthermore, if $y+t_0k=z$, then
$$ G(t_0-)-G(t_0+)=B_y(x+t_0h,0,{w}+t_0r,{v}+t_0s)k=0.$$
This establishes \eqref{conv2}.

\smallskip

\emph{Step 2. Proof of \eqref{L2max}.} Suppose that $w$ is an $A_2$ weight and let $c=2^d[w]_{A_2}$. Let $f$, $g$ be martingales such that $g$ is differentially subordinate to $f$; again, we may assume that $\|f\|_{L^2(w)}$ is finite. Fix $n\geq 0$ and let us apply \eqref{conv2} with $x=f_n$, $y=g_n$, $z=g_n^*$ (the one-sided maximal function), ${w}=w_n$, ${v}=v_n$ and $h=df_{n+1}$, $k=dg_{n+1}$, $r=dw_{n+1}$ and $s=dv_{n+1}$. Arguing as previously, we check that the application is allowed. Since $U(x,y,z,w,v)=U(x,y,y\vee z,w,v)$, we get
\begin{align*}
 U(f_{n+1},g_{n+1},g_{n+1}^*,w_{n+1},v_{n+1})&=U(f_{n+1},g_{n+1},g_{n}^*,w_{n+1},v_{n+1})\\
&\leq U(f_n,g_n,g_n^*,w_n,v_n)+I,
\end{align*}
where $\E(I|\F_n)=0$. This yields that the  sequence $(\E U(f_n,g_n,g_n^*,w_n,v_n))_{n\geq 0}$ is non-increasing, and hence by 3$^\circ$ and  then 2$^\circ$,
\begin{align*}
\frac{1}{2}\left(\E(g_n-g_n^*)^2w_n-1109^2c^2\E f_n^2v_n^{-1}\right)&\leq \E B(f_n,g_n-g_n^*,w_n,v_n)\\
&= \E U(f_n,g_n,g_n^*,w_n,v_n)\\
&\leq \E U(f_0,g_0,g_0^*,w_0,v_0)=\E B(f_0,0,w_0,v_0)\leq 0.
 \end{align*}
Repeating the arguments from the previous lemma, this implies $ \E (g-g^*)^2\leq 1109^2c^2\E f^2w,$ which is the claim.
\end{proof}

We are ready for the proof of our main estimate.

\begin{proof}[Proof of \eqref{MainTheoremL2}]
Fix a weight $w$ and a pair $(f,g)$ of martingales such that $g$ is differentially subordinate to $f$. We have $ |g|^*\leq |g-g_0|^*+|g_0|\leq (g-g_0)^*+(g_0-g)^*+|g_0|$ 
and hence 
$$ \|\,|g|^*\|_{L^2(w)}\leq \|(g-g_0)^*\|_{L^2(w)}+\|(g_0-g)^*\|_{L^2(w)}+\|g_0\|_{L^2(w)}.$$
But the martingales $g^1=g-g_0$ and $g^2=g_0-g$ are also differentially subordinate to $f$, so by \eqref{L2} and \eqref{L2max},
$$ \|g^{1*}\|_{L^2(w)}\leq \|g^1-g^{1*}\|_{L^2(w)}+\|g^1\|_{L^2(w)}\leq 2218\cdot 2^d [w]_{A_2}\|f\|_{L^2(w)}$$
and similarly for $g^2$. Furthermore, we have 
$$\|g_0\|_{L^2(w)}^2\leq \|f_0\|_{L^2(w)}^2=\E f_0^2w_0\leq [w]_{A_2}\E f_0^2v_0^{-1}\leq [w]_{A_2}\E f^2v^{-1}\leq 2^d[w]_{A_2}\|f\|_{L^2(w)}^2$$
(in the third passage we have used the estimate $w_0v_0\leq [w]_{A_2}$, in the fourth we exploited the convexity of the function $(x,v)\mapsto x^2v^{-1}$ and the last is due to $v^{-1}=w$). 
Putting all the above facts together, we get $ \|\,|g|^*\|_{L^2(w)}\leq 4437\cdot 2^d[w]_{A_2}\|f\|_{L^2(w)}$, which is the desired assertion.
\end{proof}

\begin{remark} As is well known by now, both the classical Hardy-Littlewood maximal functions $M$ and classical Calder\'on-Zygmund operators $T$ have the sharp $L^2(w)$ linear dependance on the $A_2$-characteristic. See for example \cite{Buc},  \cite{Hyt}, and the many references sited in these papers.   That is, $\|Mf\|_{L^2(w)}\leq C_1[w]_{A_2}\|f\|_{L^2(w)}$ and $\|Tf\|_{L^2(w)}\leq C_2[w]_{A_2}\|f\|_{L^2(w)}$, with $C_1$ and $C_2$ independent of $w$.  A natural conjecture from our results above is that the following inequality should hold.  
\begin{equation}\label{conj} 
\|M(Tf)-Tf\|_{L^2(w)}\leq C[w]_{A_2}\|f\|_{L^2(w)}, 
\end{equation}
with $C$ is independent of $w$.  
\end{remark} 

%\textcolor{red}{Adam:  Let me know what you think of this formulation.  What about the following (even the martingale version):  Could there be a commutator inequality of the form
%$$
%\|M(Tf)-M(Tf)\|_{L^2(w)}\leq C[w]_{A_2}\|f\|_{L^2(w)}
%$$
%There are many weighted norm commutator results with BMO functions in the literature.  But I could not locate any of the this form. Finally, I attached a paper of Lerner which goes, in some ways, in these directions. As I recall, you may have some sharp martingale versions of them.  
%}

\section{A weighted inequality for one-dimensional singular integrals}
As an application of the above results, we will establish the corresponding bounds for maximal singular integrals in dimension one. Throughout this section we assume that $K:(-\infty,0)\cup(0,\infty)\to \R$ is an odd, twice differentiable function (in the sense that $K'$ is absolutely continuous) which satisfies
\begin{equation}\label{cond1}
 \lim_{x\to \infty} K(x)=\lim_{x\to \infty}K'(x)=0
\end{equation}
and
\begin{equation}\label{cond2}
 x^3K''(x)\in L^\infty(\R).
\end{equation}
We denote by $T_K$ the associated one-dimensional singular integral operator, defined by
$$ T_Kf(x)=\operatorname{p.v.}\int_\R f(x-y)K(y)\mbox{d}y.$$
Actually, we will be mostly interested in the related maximal counterpart, given by $T_K^*f(x)=\sup_{\e>0} |T_K^\e f(x)|$, where $T_K^\e f(x)$ is the truncation at level $\e$:
$$ T_K^\e f(x)=\int_{|y|>\e} f(x-y)K(y)\mbox{d}y.$$

As shown by Vagharshakyan \cite{Va}, the operator $T_K$ can be expressed as an average of appropriate one-dimensional dyadic shifts. To recall the necessary definitions, let $G$, $H:\R\to \R$ be two functions  supported on the unit interval $[0,1]$ and given by the formulas
$$ G(x)=\begin{cases}
-1 & \mbox{if }0\leq x<1/4,\\
1 & \mbox{if }1/4\leq x< 3/4,\\
-1 & \mbox{if }3/4\leq x\leq 1
\end{cases}
\qquad \mbox{and}\qquad  H(x)=\begin{cases}
7 &\mbox{if }0<x<1/4,\\
-1 & \mbox{if }1/4\leq x< 1/2,\\
1 & \mbox{if }1/2\leq x<3/4,\\
-7 & \mbox{if }3/4\leq x\leq 1.
\end{cases}$$
For any function $f:\R\to \R$ and any interval $I=[a,b]$, we define the scaled function $f_I$ by 
$$ f_I(x)=\frac{1}{\sqrt{b-a}}f\left(\frac{x-a}{b-a}\right).$$
For any $\beta=\{\beta_l\}\in \{0,1\}^\mathbb{Z}$ and any $r\in [1,2)$ we define the dyadic grid $\mathbb{D}_{r,\beta}$ to be the following collection of intervals:
$$ \mathbb{D}_{r,\beta}=\left\{ r2^n\left([0,1)+k+\sum_{i<n} 2^{i-n}\beta_i\right)\right\}_{n\in\mathbb{Z},k\in \mathbb{Z}}.$$
We equip $\{0,1\}^\mathbb{Z}$ with the uniform probability measure $\mu$ uniquely determined by the requirement
$$ \mu(\{\beta:(\beta_{i_1},\beta_{i_2},\ldots,\beta_{i_n})=a\})=2^{-n}$$
for any $n$, any sequence $i_1<i_2<\ldots<i_n$ of integers and any $a\in \{0,1\}^n$. 

The aforementioned result of Vagharshakyan asserts the following.

\begin{theorem}
Suppose that the kernel $K$ satisfies \eqref{cond1} and \eqref{cond2}. Then there exists a coefficient function $\gamma:(0,\infty)\to \R$ satisfying
$$ \|\gamma\|_\infty\leq C\|x^2K''(x)\|_\infty$$
such that
\begin{equation}\label{defK}
 K(x-y)=\int_{\{0,1\}^\mathbb{Z}}\int_1^2 \sum_{I\in \mathbb{D}_{r,\beta}} \gamma(|I|)H_I(x)G_I(y)\frac{\mbox{d}r}{r}\mbox{d}\mu(\beta)
\end{equation}
for all $x\neq y$. Here $C$ is some absolute constant and the series on the right is absolutely convergent almost everywhere.
\end{theorem}

In other words, we see that the singular integral $T_K$ can be expressed as an average of the Haar shift operators
$$ T_{r,\beta}f=\sum_{I\in \mathbb{D}_{r,\beta}} \gamma(|I|)\langle f,G_I\rangle H_I(x).$$
%It will be convenient for us to split each such operator into its odd and even part:
%$$ T_{r,\beta}^{\text{odd}}f=\sum \gamma(|I|)\langle f,G_I\rangle H_I(x), \qquad T_{r,\beta}^{\text{even}}f=\sum \gamma(|I|)\langle f,G_I\rangle H_I(x),$$
%in which the summation runs over all $I$ in $\mathbb{D}_{r,\beta}$ with $\log_2|I|$ odd or even, respectively. 
We will now study the relation between such operators and differentially subordinate martingales. For any dyadic subinterval of $[0,1)$, let $h_I$ be the associated normalized Haar function given by $h_I=|I|^{-1/2}(\chi_{I_-}-\chi_{I_+})$, where $I_-$ and $I_+$ stand for the left and the right half of $I$. Note that $G_I=2^{-1/2}(-h_{I_-}+h_{I_+})$. 

\begin{lemma}\label{lemmam}
For any $a_1,\,a_2,\,a_3\in \R$, let $f=a_1h_{[0,1)}+a_2h_{[0,1/2)}+a_3h_{[1/2,1)}$ and $\displaystyle g=(-a_2+a_3)\cdot 2^{-1/2}H$. Then there is a filtration $(\F_0,\F_1,\F_2)$ of the probability space $([0,1),\mathcal{B}(0,1),|\cdot|)$ with $\F_0=\{[0,1),\emptyset\}$ and $\F_2=\sigma([0,1/4),[1/4,1/2),[1/2,3/4))$, such that the martingale $(g_n)_{n=0}^2$ associated with $g$ is differentially subordinate to the martingale $(56f_n)_{n=0}^2$ associated with $56f$.
\end{lemma}
\begin{proof}
It will be more convenient for us to replace $a_2$, $a_3$ above with $a_2\cdot 2^{-1/2}$ and $a_3\cdot 2^{-1/2}$. Then the function $f$ takes values in the set $\{a_1+a_2,a_1-a_2,-a_1+a_3,-a_1-a_3\}$. If all elements of this set are bigger in absolute value than $|-a_2+a_3|/4$, then we take $\F_1=\F_2$. For such a choice, we have, with probability $1$, 
$$|df_1|=|f|\geq |-a_2+a_3|/4=\|g\|_\infty/14=\|dg_1\|_\infty/14$$
and $df_2=df_3=dg_2=dg_3=0$, so the differential subordination holds true. Now, suppose that there is a set $C\in \{[0,1/4),[1/4,1/2),[1/2,3/4),[3/4,1)\}$ on which $|f|\leq  |-a_2+a_3|/4$. Let $m=\max\{|a_1+a_2|,|a_1+a_3|\}$, $M=\max\{|a_1-a_2|,|a_1-a_3|\}$. By the triangle inequality,  
\begin{equation}\label{boundmM}
M,m\geq |a_2-a_3|/2.
\end{equation}
Let $A$, $B\in \{[0,1/4),[1/4,1/2)$, $[1/2,3/4),[3/4,1)\} $ be disjoint sets such that $|f|=m$ on $A$ and $|f|=M$ on $B$. By \eqref{boundmM} and the definition of the set $C$, we see that $A$, $B$, $C$ are pairwise disjoint; set $D=[0,1]\setminus (A\cup B\cup C)$ and define $\F_1=\sigma(A\cup C)$. Since $|A\cup C|=1/2$, the random variable $df_1=f_1$ takes values in a set of the form $\{b_0,b_0\}$. Using \eqref{boundmM}, we have
$$ |b_0|\geq \frac{|f||_A-|f||_C}{2}\geq \frac{m-|-a_2+a_3|/4}{2}\geq \frac{|-a_2+a_3|}{8}\geq \frac{\|g\|_\infty}{28}=\frac{\|dg_1\|_\infty}{28}$$
almost surely (in the last passage, we have used the equality $g_0=0$). Next, the variable $df_2$ restricted to $A\cup C$ also takes values in a set of the form $\{-b_1,b_1\}$ (since $|A|=|C|$), and therefore, with probability $1$,
$$ |b_1|\geq |f||_A-|b_0|\geq m-\left(\frac{m+|-a_2+a_3|/4}{2}\right)\geq \frac{\|g\|_\infty}{28}\geq \frac{\|dg_2\|_\infty}{28}.$$
Similarly, $df_2$ restricted to $B\cup D$ takes values in a set of the form $\{-b_2,b_2\}$, and analogous analysis to that above gives $ |b_2|\geq \|dg_2\|_\infty/56$ 
almost surely.
\end{proof}

By the self-similarity of the Haar system, the above lemma gives the following.

\begin{corollary}\label{corolla}
Suppose that $f$ is a function on $[0,1)$ satisfying $\int_0^1 f=0$, let $\mathcal{D}$ be the lattice of dyadic subintervals of $[0,1)$ and let $\gamma=(\gamma_I)_{I\in \mathcal{D}}$ be an arbitrary sequence satisfying $\|\gamma\|_\infty\leq 1$. Assume in addition that either all $\gamma_I$ with $\log_2 |I|$ odd are zero, or all $\gamma_I$ with even $\log_2 |I|$ vanish. Then there is a filtration on $([0,1),\mathcal{B}(0,1),|\cdot|)$ such that the adapted martingale $(g_n)_{n\geq 0}$ generated by $g=T^\gamma(f)=\sum_{I\in \mathcal{D}} \gamma_I \langle f,G_I\rangle H_I$ is differentially subordinate to the martingale $(56f_n)_{n\geq 0}$  generated by $56f$.
\end{corollary}
\begin{proof}[Proof (sketch)] 
We will show the claim for ``even'' shifts only. Let us start with setting $\F_0=\{\emptyset,\Omega\}$. The remaining $\sigma$-algebras of the filtration are constructed in pairs: $\{\F_1,\F_2\}$, $\{\F_3,\F_4\}$, and so on, furthermore, for each $n$ we have the equality $\F_{2n}=\sigma([0,2^{-2n}),[2^{-2n},2\cdot 2^{-2n}),[2\cdot 2^{-2n},3\cdot 2^{-2n}),\ldots,[1-2^{-2n},1))$. To explain the construction of $\F_{2n+1}$ for a given $n\geq 0$, fix an atom $I$ of $\F_{2n}$: $I=[k\cdot 2^{-2n},(k+1)\cdot 2^{-2n})$ for some $0\leq k\ \leq 2^{2n}-1$. Let us apply the previous lemma conditionally on $I$ with the function $ \mathtt{f}^I:=\langle f,h_I\rangle h_I+\langle f,h_{I_-}\rangle h_{I_-}+\langle f,h_{I_+}\rangle h_{I_+}$.  As the result, we get that the interval $I$ can be filtered with the use of three $\sigma$-algebras $\F_0(I)=\{\emptyset,I\}$, $\F_1(I)$, $\F_2(I)$  such that the martingale corresponding to $\mathtt{g}^I=\langle f,G_I\rangle H_I$ is differentially subordinate to the martingale induced by $56\mathtt{f}^I$. The crucial observation is that the atoms of the $\sigma$-algebra $\F_1(I)$ are built from two dyadic quarters (grandchildren) of $I$. Therefore, if we set
$$ \F_{2n+1}=\sigma(\F_1(I):I\mbox{ is an atom of }\F_{2n}),$$
then $\F_0$, $\F_1$, $\F_2$, \ldots are properly ordered (i.e., they form a filtration). To check the differential subordination, note first that we have $df_0=dg_0=0$. Furthermore, if $n$ is a nonnegative integer and $I$ is an atom of $\F_{2n}$, then 
$$|df_{2n+1}|=|d\mathtt{f}^I_1|\geq |d\mathtt{g}^I_1|\geq |dg_{2n+1}|$$
on $I$ (in the last line we use the boundedness assumption on the sequence $\gamma$ appearing in the definition of $T^\gamma f$);  the inequality $|df_{2n+2}|\geq |dg_{2n+2}|$ is proved similarly. 
\end{proof}

%(\textcolor{red}{Adam: I see the argument here.  But at each stage we apply the theorem with dyadic weights. Is it the point that since  $w\in A_p$ (sup over all intervals)  then on any subinterval $w \in dyadic\,A_p$ and the dyadic characteristic is bounded, uniformly, by the $A_p$ characteristic?})
We are ready to establish the main result of this section.

\begin{theorem}
For any $1<p<\infty$ and any kernel $K$ satisfying the above assumptions and any $A_p$ weight $w$ on the real line, we have the estimate
$$ \|T_K^*f\|_{L^p(w)}\leq \tilde{C}_p[w]_{A_p}^{\max\{1/(p-1),1\}}\|f\|_{L^p(w)}.$$
The exponent $\max\{1/(p-1),1\}$ is the best possible for each $p$.
\end{theorem}
\begin{proof}
The optimality of the exponent is clear: it is already the best in the non-maximal case (for the Hilbert transform, corresponding to the choice $K(x)=1/x$). By standard approximation, we may and do assume that $f$ is supported on some interval $J$. Furthermore, we may assume that  $\int_\R f=0$, adding to $f$, if necessary, the ``correction function'' of the form $a\chi_{[M,M+1]}$ and then sending $M$ to infinity. 
Fix $\e>0$ and two real numbers $x$, $y$ satisfying $|x-y|>\e$. Since both $G_I$ and $H_I$ are supported on $I$, we see that $H_I(x)G_I(y)=0$ if $|I|<\e/2$ and we may rewrite the identity \eqref{defK} in the form
$$ K(x-y)=\int_{\{0,1\}^\mathbb{Z}}\int_1^2 \sum_{I\in \mathbb{D}_{r,\beta} \atop |I|\geq \e/2} \gamma(|I|)H_I(x)G_I(y)\frac{\mbox{d}r}{r}\mbox{d}\mu(\beta).$$
Therefore we have
\begin{align*}
 |T^\e f(x)|&=\left|\int_{\{0,1\}^\mathbb{Z}}\int_1^2 \sum_{I\in \mathbb{D}_{r,\beta} \atop |I|\geq \e/2} \gamma(|I|)\langle f,G_I\rangle H_I(x)\frac{\mbox{d}r}{r}\mbox{d}\mu(\beta)\right|\\
 &\leq \int_{\{0,1\}^\mathbb{Z}}\int_1^2 \left|\sum_{I\in \mathbb{D}_{r,\beta} \atop |I|\geq \e/2} \gamma(|I|)\langle f,G_I\rangle H_I(x)\right|\frac{\mbox{d}r}{r}\mbox{d}\mu(\beta)\\
 &\leq \int_{\{0,1\}^\mathbb{Z}}\int_1^2 \sup_{\eta>0}\left|\sum_{I\in \mathbb{D}_{r,\beta} \atop |I|\geq \eta} \gamma(|I|)\langle f,G_I\rangle H_I(x)\right|\frac{\mbox{d}r}{r}\mbox{d}\mu(\beta)
\end{align*}
and hence, by Jensen's inequality,
\begin{equation}\label{maximal}
\begin{split}
 &\|T^*f\|_{L^p(w)}^p\\
&\leq (\ln 2)^{p-1}\int_{\{0,1\}^\mathbb{Z}}\int_1^2 \int_{\R}\sup_{\eta>0}\left|\sum_{I\in \mathbb{D}_{r,\beta} \atop |I|\geq \eta} \gamma(|I|)\langle f,G_I\rangle H_I(x)\right|^pw(x)\mbox{d}x\frac{\mbox{d}r}{r}\mbox{d}\mu(\beta).
\end{split}
\end{equation}
Now fix $r\in [1,2)$ and $\beta\in \{0,1\}^\mathbb{Z}$. Almost surely (with respect to the measure $\mu$), there is an interval in $\mathbb{D}_{r,\beta}$ which entirely contains $J$, the support of $f$. Since $\int_\R f=0$, we see that $\langle f,G_J\rangle=0$ for $J\in \mathbb{D}_{r,\beta}$ of sufficiently large measure; therefore
$$\sum_{I\in \mathbb{D}_{r,\beta} \atop |I|\geq \eta} \gamma(|I|)\langle f,G_I\rangle H_I(x)=\sum_{I\in \mathbb{D}_{r,\beta},\,I\subseteq J_{r,\beta}, \atop |I|\geq \eta} \gamma(|I|)\langle f,G_I\rangle H_I(x)$$
for some $J_{r,\beta}\in \mathbb{D}_{r,\beta}$. Let us split the sum on the right into two sums, corresponding to those $I$, for which $\log_2 (r^{-1}|I|)$ is odd, and those $I$, for which $\log_2(r^{-1}|I|)$ is even. Let us go back to Corollary \ref{corolla} (rescaled, with $[0,1)$  replaced by $J_{r,\beta}$) and the martingales $56f$ and $g$ studied there. The expression
$$ \sum_{I\in \mathbb{D}_{r,\beta},\,I\subseteq J_{r,\beta}, \atop |I|\geq \eta,\log_2(r^{-1}|I|)\text{ odd}} \gamma(|I|)\langle f,G_I\rangle H_I,$$
considered as a function on $J$, is one of the variables of the martingale $g$.  Consequently, the supremum
$$ x\mapsto \sup_{\eta>0} \left|\sum_{I\in \mathbb{D}_{r,\beta},\,I\subseteq J_{r,\beta}, \atop |I|\geq \eta,\log_2(r^{-1}|I|)\text{ odd}} \gamma(|I|)\langle f,G_I\rangle H_I(x)\right|$$
is bounded from above by the maximal function of $g$ and the inequality \eqref{MainTheorem} implies
$$ \int_\R \sup_{\eta>0} \left|\sum_{I\in \mathbb{D}_{r,\beta},\,I\subseteq J_{r,\beta}, \atop |I|\geq \eta,\log_2(r^{-1}|I|)\text{ odd}} \gamma(|I|)\langle f,G_I\rangle H_I(x)\right|^p\mbox{d}x\leq C_p^p[w]_{A_p}^{\max\{p/(p-1),p\}}\|56f\|_{L^p(w)}^p.$$
The same estimate holds true for the `even' sum (where the summands correspond to the intervals $I$ with even $\log_2(r^{-1}|I|)$). Consequently,
$$ \int_\R \sup_{\eta>0} \left|\sum_{I\in \mathbb{D}_{r,\beta},\,I\subseteq J_{r,\beta}, \atop |I|\geq \eta} \gamma(|I|)\langle f,G_I\rangle H_I(x)\right|^p\mbox{d}x\leq C_p^p[w]_{A_p}^{\max\{p/(p-1),p\}}\|112f\|_{L^p(w)}^p.$$
It remains to plug this bound into \eqref{maximal} to get the claim.
\end{proof}
}

We conclude by proving the maximal inequality formulated in the introductory section. 

\begin{proof}[Proof of \eqref{MainTheorem}]
By the extrapolation argument, it is enough to show the estimate for $p=2$ only. Fix $c>1$ and consider the function $\mathbb{B}:\{(x,y,z,w,v)\in \R^3\times \mathcal{D}_c: x\leq z\}\to \R$ given by $\mathbb{B}(x,y,z,w,v)=B(x,y-z,w,v).$ This new object enjoys the properties listed in Remark \ref{extension} above (in 2$^\circ$, it majorizes $\mathbb{G}(x,y,z,w,v)=G(x,y-z,w,v)=\kappa ((y-z)^{2}{w}-C^{2}c^{2}x^{2}{v}^{-1})$); in particular, we have $ \mathbb{B}_z(x,y,y,w,v)=B_y(x,0,w,v)=0$, by the symmetry of $B$. Hence we obtain
$$ \E (Y_t^*-Y_t)^2W_t\leq C^2c^2\E X_t^2V_t^{-1}$$
for any $X,\,Y$ such that $Y$ is a stochastic integral of $X$ and any pair $(W,V)$ associated with an $A_2$ weight of characteristic not exceeding $c$. Letting $t\to \infty$ and using some standard limiting arguments (and the equality $V_\infty^{-1}=W_\infty$), we get
$$ \|Y^*-Y\|_{L^2(W)}\leq C[W]_{A_2}\|X\|_{L^2(W)}.$$ 
It remains to use the fact that the two-sided maximal function $|Y|^*$ satisfies $|Y|^*\leq |Y^*|+|(-Y)^*|$, which implies
$$ \||Y|^*\|_{L^2(W)}\leq \|Y^*\|_{L^2(W)}+\|(-Y)^*\|_{L^2(W)}\leq 4C[W]_{A_2}\|X\|_{L^2(W)}.$$
The proof is complete.
\end{proof}


\begin{thebibliography}{10}
\bibitem{AIPS0}{K. Astala, T. Iwaniec, I. Prause and E. Saksman, {\it Burkholder integrals, Morrey's problem and quasiconformal mappings}, 
J. Amer.  Math.  Soc. \textbf{25} (2012), 507--531.}

\bibitem{AIPS}{K. Astala, T. Iwaniec, I. Prause and E. Saksman, {\it A hunt for sharp 
$L^p$-estimates and rank-one convex variational integrals}, Filomat \textbf{29} (2015), 245--261.}

\bibitem{BMS}{A. Baernstein II and S. J. Montgomery-Smith, {\it Some conjectures about integral means of $\partial f$ and $\bar{\partial}f$}, in Complex Analysis and Differential Equations, Acta Univ. Upsaliensis Skr. Uppsala Univ. C. Organ. Hist. \textbf{64}, Uppsala Univ., Uppsala, 1999, pp. 92--109.}

\bibitem{BB} R. Ba{\~n}uelos and K. Bogdan, \textit{L\'evy processes and Fourier multipliers}, J. Funct. Anal. 250 (2007) no. 1, 197-213.

%\bibitem{BMH} R. Ba{\~n}uelos and P. J. M\'endez-Hernandez, \textit{Space-time Brownian motion and the Beurling-Ahlfors transform}, Indiana Univ. Math. J. 52 (2003), no. 4, 981-990.

\bibitem{BO11}{R. Ba\~nuelos, A. Os\k ekowski, {\it Inequalities for Fourier multipliers related to Astala's theorem}, Adv. Math. \textbf{283} (2015), 275--302.}
\bibitem{BO12}{R. Ba\~nuelos, A. Os\k ekowski, {\it Sharp martingale inequalities and applications to Riesz transforms on manifolds, Lie groups and Gauss space}, Journal of Functional Analysis \textbf{269} (2015), 1652--1713.}
%\bibitem{BO1} R. Ba{\~n}uelos, A. Os\k{e}kowski, \textit{Sharp Weighted L2 inequalities for square functions}, to appear in Trans. Amer. Math. Soc. (Electronically at http://dx.doi.org/10.1090/tran/7056)

\bibitem{BO2} Rodrigo Ba{\~n}uelos and Adam Os\k{e}kowski, \textit{Weighted square function estimates}, to appear in Bull. Sci. Math.,   
arXiv:1711.08754v1 [math.PR] 23 Nov 2017.

%\bibitem{BW} R. Ba{\~n}uelos and G. Wang, \textit{Sharp inequalities for martingales with applications to the Beurling-Ahlfors and Riesz transforms}, Duke Math. J. 80 (1995), no. 3, 575-600.
\bibitem{Be}{R. Bellman, Dynamic programming, Reprint of the 1957 edition, Princeton Landmarks in Mathematics. Princeton University Press, Princeton, NJ, 2010.}
\bibitem{Buc} S.M. Buckley, {\it Estimates for operator norms on weighted spaces and reverse Jensen inequalities,} Trans. Amer. Math. Soc. 340 (1) (1993) 253-272.

\bibitem{Bu1} D. L. Burkholder, \textit{Boundary value problems and sharp inequalities for martingale transforms}, Ann. Probab. 12 (1984), 647-702.

\bibitem{Bu2} D. L. Burkholder, \textit{A sharp and strict Lp-inequality for stochastic integrals}, Ann. Probab.15
(1987), 268--273.

\bibitem{Bu3} D. L. Burkholder, \textit{A proof of Pe\l czy\'nski's conjecture for the Haar system}, Studia Math. 91 (1988), 79-83.

\bibitem{Bu4} D. L. Burkholder, \textit{Explorations in martingale theory and its applications}, \'Ecole d'Ete de Probabiliti\'es de Saint-Flour XIX-1989, pp. 1-66, Lecture Notes in Math., 1464, Springer,
Berlin, 1991.

\bibitem{DM} C. Dellacherie and P. A. Meyer, \textit{Probabilities and potential} B, North-Holland, Amsterdam,1982.

\bibitem{DP} K. Domelevo and S. Petermichl, {\it Differential subordination under change of law}, Ann. Probab. \textbf{47} (2019), 896--925.

\bibitem{D} J. Duoandikoetxea, \textit{Extrapolation of weights revisited: new proofs and sharp bounds}, J. Funct. Anal. 260 (2011), 1886-1901

%\bibitem{GMSS} S. Geiss, S. Montgomery-Smith and E. Saksman, \textit{On singular integral and martingale transforms}, Trans. Amer. Math. Soc. 362 No. 2 (2010), 553-575.
\bibitem{I1}{T. Iwaniec, {\it Extremal inequalities in Sobolev spaces and quasiconformal mappings}, Z. Anal. Anwendungen \textbf{1} (1982), 1--16.}
\bibitem{I2}{T. Iwaniec, {\it Nonlinear Cauchy-Riemann operators in $\R^n$}, Trans. Amer. Math. Soc.354(2002), 1961--1995.}

\bibitem {IK} M. Izumisawa and N. Kazamaki, \textit{Weighted norm inequalities for martingales}, Tohoku Math. Journ. 29(1977), 115-124.
\bibitem{Hyt} T.P. Hyt\"onen, {The sharp weighted bound for general Calder\'on-Zygmund operators,} Annals of Mathematics,  175 (2012), 1473-1506
\bibitem{L}{M. T. Lacey, K. Moen, C. P\'erez and R. H. Torres, {\it Sharp weighted bounds for fractional integral operators}, J. Funct. Anal. \textbf{259} (2010), pp. 1073--1097.}
\bibitem{Le1}{A. Lerner, {\it Sharp weighted norm inequalities for Littlewood-Paley operators and singular integrals}, Adv. Math. \textbf{226} (2011), 3912--3926.}
\bibitem{Mu}{B. Muckenhoupt, {\it Weighted norm inequalities for the Hardy maximal function}, Trans. Amer. Math. Soc.  {165} (1972), 207--226.}

\bibitem{NT} F.L. Nazarov and S.R. Treil, \emph{The hunt for a Bellman function: applications to estimates for singular integral operators and to other classical problems of harmonic analysis,} St. Petersburg Math. J. {\bf 8} (1997), 721--824. 

\bibitem{NTV}{F. L. Nazarov, S. R. Treil and A. Volberg, {\it The Bellman functions and two-weight inequalities for Haar multipliers}, J. Amer. Math. Soc., \textbf{12} (1999), pp. 909--928.}

\bibitem {O1} A. Os\k{e}kowski, \textit{Sharp martingale and semimartingale inequalities}, Monografie Matematyczne 72, Birkh{\"a}user, 2012.

\bibitem{P} G. Peskir, \textit{A Change-of-Variable Formula with Local Time on Surfaces}. In: Donati-Martin C., \'Emery M., Rouault A., Stricker C. (eds) S\'eminaire de Probabilit\'es XL. Lecture Notes in Mathematics, vol 1899. Springer, Berlin, Heidelberg, 2007.

\bibitem{PV}{S. Petermichl and A. Volberg, {\it Heating of the Ahlfors-Beurling operator: weakly quasiregular maps on the plane are quasiregular}, Duke Math. J. \textbf{112} (2002), 281--305.}

\bibitem{Va}{A. Vagharshakyan, {\it Recovering singular integrals from Haar shifts},  Proc. Amer. Math. Soc., 138(12):4303--4309, 2010.}

\bibitem{W} G.Wang, \textit{Differential subordination and strong differential subordination for continuous time martingales and related sharp inequalities}, Ann. Probab. 23 (1995), 522-551.

\bibitem {Wi} J. Wittwer, \textit{A sharp estimate on the norm of the martingale transform}, Math. Res. Lett. 7(1), 1-12 (2000).

\end{thebibliography}
\end{document}